\documentclass[a4paper]{article}

\usepackage[english]{babel}
\usepackage{amsmath}
\title{Syntactic categories for Nori motives}

\author{Luca Barbieri-Viale \and Olivia Caramello \\  \and Laurent Lafforgue}
\date{}

\usepackage[all]{xy}
\usepackage[T1]{fontenc}
\usepackage{tikz,color,makeidx, enumerate}
\usetikzlibrary{matrix,arrows}

\newcommand{\cC}{\mathcal{C}}
\newcommand{\Hom}{\operatorname{Hom}}

\newcommand{\Rmod}{\text{$R$-{\sf mod}}}
\newcommand{\Rpmod}{\text{$R'$-{\sf mod}}}
\newcommand{\into}{\hookrightarrow}
\newcommand{\cA}{\mathcal{A}}
\newcommand{\by}[1]{\stackrel{#1}{\rightarrow}}
\newcommand{\End}{\operatorname{End}} 
\newcommand{\Th}{\operatorname{Th}} 
\newcommand{\colim}[1]{\mathop{\rm
Colim}_{\buildrel\over{#1}}}
\newcommand{\longby}[1]{\stackrel{#1}{\longrightarrow}}

\mathcode`\<="4268  
\mathcode`\>="5269  
\mathcode`\.="313A  
\mathchardef\semicolon="603B 
\mathchardef\gt="313E
\mathchardef\lt="313C

\newcommand{\cod}
 {{\rm cod}}

\newcommand{\Cont}
 {{\bf Cont}}

\newcommand{\dom}
 {{\rm dom}}

\newcommand{\empstg}
 {[\,]}

\newcommand{\epi}
 {\twoheadrightarrow}

\newcommand{\hy}
 {\mbox{-}}

\newcommand{\im}
 {{\rm Im}}

\newcommand{\imp}
 {\!\Rightarrow\!}

\newcommand{\Ind}[1]
 {{\rm Ind}\hy #1}

\newcommand{\mono}
 {\rightarrowtail}
 
\newcommand{\name}[1]
 {\mbox{$\ulcorner #1 \urcorner$}}

\newcommand{\ob}
 {{\rm ob}}

\newcommand{\op}
 {^{\rm op}}

\newcommand{\Set}
 {{\bf Set }}

\newcommand{\Sh}
 {{\bf Sh}}

\newcommand{\sh}
 {{\bf sh}}

\newcommand{\Sub}
 {{\rm Sub}}

\usepackage{mathrsfs}
\usepackage{amsmath,amsthm}
\usepackage{amssymb}

\newtheorem{theorem}{Theorem}[section]
\theoremstyle{definition}
\newtheorem{definition}[theorem]{Definition}
\newtheorem{lemma}[theorem]{Lemma}

\newtheorem{corollary}[theorem]{Corollary}
\newtheorem{remark}[theorem]{Remark}
\newtheorem{remarks}[theorem]{Remarks}

\begin{document}

\maketitle

\begin{abstract}
	We give a new construction, based on categorical logic, of Nori's $\mathbb Q$-linear abelian category of mixed motives associated to a cohomology or homology functor with values in finite-dimensional vector spaces over $\mathbb Q$.
	This new construction makes sense for infinite-dimensional vector spaces as well, so that it associates a $\mathbb Q$-linear abelian category of mixed motives to any (co)homology functor, not only Betti homology (as Nori had done) but also, for instance, $\ell$-adic, $p$-adic or motivic cohomology.
	We prove that the $\mathbb Q$-linear abelian categories of mixed motives associated to different (co)homology functors are equivalent if and only a family (of logical nature) of explicit properties is shared by these different functors. The problem of the existence of a universal cohomology theory and of the equivalence of the information encoded by the different classical cohomology functors thus reduces to that of checking these explicit conditions.
	
\end{abstract}

\tableofcontents

\section*{Introduction}

Nori's construction in the theory of motives (\cite{NM}) starts with a representation $T$ of an arbitrary diagram $D$ (which is not necessarily a category) in the category $k\textrm{-vect}_{\textrm{f}}$ of finite-dimensional vector spaces over a field $k$ or, more generally, the category $\Rmod_{\rm f}$ of finite-type modules over a Noetherian ring $R$. It constructs a factorisation of $T$ through a faithful exact $k$-linear (or $R$-linear) functor $F$ from an abelian $k$-linear (or $R$-linear) category ${\cal C}_{T}$ to the category $k\textrm{-vect}_{\textrm{f}}$ or $\Rmod_{\rm f}$ which is universal for this factorisation property.
    
This extremely general construction was applied by Nori to the relative Betti homology functor of pairs of finite-type schemes over a subfield of $\mathbb C$. This means that the objects of the diagram $D$ are triples $(X,Y,i)$ consisting in such a scheme $X$, a closed subscheme $Y$ of $X$ and a non-negative integer $i$. The arrows are of two different types: firstly there are arrows $(X,Y,i) \to (X',Y',i)$ associated to geometric morphisms $(X,Y) \to (X',Y')$ and secondly there are arrows $(X,Y,i)\to (Y,Z,i-1)$ associated to triples consisting in a scheme $X$, a closed subscheme $Y$ of $X$ and a closed subscheme $Z$ of $Y$. The representation $T$ associates to any such triple $(X,Y,i)$ the $i$-th Betti homology group of $X$ relatively to $Y$.
    
J. Ayoub and L. Barbieri-Viale have proved in \cite{ABV} that the abelian category constructed in this way from the restriction of the diagram $D$ to triples $(X, Y, i)$ with $i$ at most $1$ is equivalent to the abelian category of Deligne $1$-motives with torsion (for $i=0$ to Artin motives). 
    
On the other hand, A. Huber and S. M\"uller-Stach have proved in \cite{HMS} that the spectrum of Kontsevich's algebra of formal periods is a torsor under the motivic Galois group associated to Nori's category of mixed motives.
    
These two different results are strong indications that Nori's construction is very interesting.

In order to associate his universal abelian category ${\cal C}_{T}$ to a representation $T$ on a diagram $D$ as above, Nori first considers the case when $D$ is a finite diagram. In that case, the endomorphism ring of $T$ is an algebra of finite dimension over the coefficient field $k$ (or of finite-type as a module over the Noetherian coefficient ring $R$) and ${\cal C}_{T}$ can be constructed as the category of modules over this algebra which are finite-dimensional over $k$ (or of finite-type over $R$).

When $D$ is infinite, the category ${\cal C}_{T}$ is constructed as the filtered colimit of the categories associated to all the finite (full) subdiagrams of $D$.

The construction presented in this paper is entirely different, even though it solves the same universal factorisation problem, and it makes sense for any representation of a diagram $D$ in the category $\Rmod$ of modules over an arbitrary ring $R$.
    It uses the language and a few results of first-order categorical logic which are summarised or proved in the first part of the paper.

The first step in the construction consists in associating to the representation $T$ the so-called regular theory of $T$. The language of this theory consists in sorts associated to the objects of the diagram, function symbols associated to the arrows of the diagram as well as to the $R$-linear structure operations (addition and multiplication by elements of $R$) and constants corresponding to the zero elements of such module structures. The axioms of the theory consists of all the regular sequences written in this language which are satisfied by the representation $T$.
    
The second step in the construction consists in associating to the regular theory of $T$ its so-called syntactic category. The objects of the syntactic category of the regular theory of $T$ are all the (regular) formulas written in the language of $T$; the morphisms between two objects, that is between two such formulas, are formulas which are provably functional with respect to these two formulas, considered up to provable equivalence in the theory.

The syntactic category of the regular theory of $T$ is regular, as it is associated to a regular theory, and it is additive and $R$-linear by construction, but it is not abelian as it lacks quotients.
    
The third step consists in replacing this syntactic regular category by its effectivization, a construction which formally adds quotients of equivalence relations in a way which admits a fully explicit description. The main theorem of this paper (Theorem \ref{thm:main}) is that this category is abelian, $R$-linear and solves the universal factorisation problem.

The generality of this construction allows to associate an $R$-linear abelian category of ``mixed motives'' to any homology or cohomology functor with coefficients in a field or a ring which contains $R$. For instance, each of the usual cohomology or homology theories, such as Betti, $\ell$-adic, $p$-adic, De Rham, crystalline or motivic cohomology or homology, gives rise in this way to a $\mathbb Q$-linear abelian category of mixed motives.
    
As the construction proposed in this paper is explicit, it allows to give necessary and sufficient conditions for two representations defined on the same diagram $D$ to give rise to equivalent categories of mixed motives. In fact, it means that the regular theories of these representations are the same, i.e. that any regular sequence is verified by one of the representations if and only if it is verified by the other. Concretely it means that inclusion relation between ``definable'' (by suitable combinations of arrows coming from the diagram $D$ and the chosen coefficient ring $R$) subspaces of the image spaces are the same for the two representations (see Theorem \ref{thm:independence}(ii)).

Let us present in a few words the contents of the paper.
    The necessary ingredients from categorical logic are gathered in the first section, while the second section is devoted to the construction and study of the universal factorising abelian category.
    Section 2.1 reviews Nori's construction.
    Section 2.2 constructs the desired abelian category and proves that it is a solution of the universal factorisation problem. In fact, even under Nori's hypotheses, Nori's category and ours are equivalent but not isomorphic and our category verifies a slightly better factorisation property, replacing a factorisation up to isomorphism of functors with a factorisation with strict equality of functors.
    Section 2.3 gives conditions which allow to realise our category as a subcategory of bigger categories and provides a concrete description of the objects and morphisms of our category.
    Section 2.4 applies this description in order to understand when two categories associated to two different representations are equivalent.
    Section 2.5 studies concretely, in the case of finite-dimensional spaces over a field, the equivalence between our category and the already known representation of Nori's category as finite-dimensional comodules over some coalgebra. Even in that case, we get a new concrete description of the objects and morphisms of that category. 
    In section \ref{sec:applications} we apply to motives our general construction, discussing it also in relation to the classical construction of Nori motives. We deduce in particular a general criterion for the existence of categories of mixed motives.

    Future work in connection with the construction presented in section $2$ will explore the question of change of coefficient rings, the question of change of diagrams $D$, the tensor structure induced from K\"unneth-type formulas for representations $T$, and other aspects.

   The first named author and the third named author want to recognize here that the construction and the results of this paper are due to the second named author. They originated in a question raised by the first named author on the possibility of reinterpreting Nori's construction in terms of the theory of classifying toposes, which the third named author had talked to him about. The second named author has benefited from many hours of conversations on algebraic geometry with the third named author and the first named author.

\section{Syntactic categories}\label{sec:preliminaries}

The general theory of syntactic categories of logical theories has been introduced in \cite{MR}. For the classical results reviewed in the following sections we refer, besides \cite{MR}, to the standard references on topos theory such as \cite{GrothendieckT}, \cite{MM} and \cite{El}. A succinct overview of the relevant logical and topos-theoretic background is contained in \cite{OC}. 

\subsection{Some categorical preliminaries}

A word on terminology: as it is usual, when we say that a category possesses certain kinds of limits or colimits, we mean that the category comes equipped with a \emph{canonical choice} of limit or colimit for diagrams of the appropriate kind. On the other hand, when we speak of functors preserving certain kinds of limits or colimits, we do not require them to preserve such canonical choices of limits or colimits.  

A functor is said to be \emph{conservative} if it reflects isomorphisms. A balanced category is a category in which arrows which are both monomorphisms and epimorphisms are isomorphisms. Any abelian category is balanced. A faithful functor on a balanced category is conservative. Conversely, any conservative functor on a category with equalizers which preserves equalizers is faithful; indeed, for any pair of parallel arrows $(f, g)$, the condition $F(f)=F(g)$ is equivalent to the condition that $F(e)$ be an isomorphism, where $e$ is the equaliser of $f$ and $g$. So in the context of abelian categories and exact functors between them, conservativity is equivalent to faithfulness.

Given a faithful functor $F:{\cal A}\to {\cal B}$, it is natural to wonder under which conditions $F$ defines an equivalence of categories between ${\cal A}$ and a subcategory of $\cal B$. Notice that, unless the functor $F$ is full on objects (in the sense of the following definition), it might be hard to prove fullness. Indeed, if $F(a)=F(a')$ without there being an arrow $a\to a'$, any arrows of the form $F(g)\circ F(f)$, where $f:a''\to a$ and $g:a'\to a'''$ would belong to any subcategory of $\cal B$ containing the objects and arrows in the image of the functor $F$, without the composition $g\circ f$ being defined in $\cal A$. 

\begin{definition}
A functor $F:{\cal A}\to {\cal B}$ is \emph{full on objects} if for any objects $a$ and $a'$ such that $F(a)=F(a')$ there exists an arrow $f:a\to a'$ in $\cal A$ such that $F(f)=1_{F(a)}$.
\end{definition}

The notion of functor which is full on objects is particularly relevant for faithful functors; indeed, if $F$ is faithful then the arrow $f$ in the definition is necessarily an isomorphism, and, as shown by the following lemma, there is a well-defined subcategory of $\cal B$ which is equivalent to $\cal A$ via $F$.

\begin{lemma}\label{lemma:fullonobjects}
Let $F:{\cal A}\to {\cal B}$ be a faithful functor which is full on objects. Then ${\cal A}$ is equivalent, via $F$, to a subcategory $\im(F)$ of $\cal B$ defined as follows: the objects of $\im(F)$ are the objects of $\cal B$ of the form $F(a)$ for $a\in {\cal A}$, while the arrows $F(a)\to F(a')$ are the arrows of $\cal B$ of the form $F(f)$ where $f:a\to a'$ is an arrow in $\cal A$. 
\end{lemma}

\begin{proof}
The only thing to check is that the subcategory $\im(F)$ of $\cal B$ is well-defined. But this clearly follows from the condition of fullness on objects.
\end{proof}

\begin{remark}
The condition that $F$ be full on objects is, unlike that of reflecting equalities of objects, necessary for $F$ to be part of an equivalence of categories onto a subcategory of $\cal B$.
\end{remark}

\subsection{Regular categories}

A \emph{regular category} is a small category with finite limits in which images of arbitrary arrows exist and are stable under pullback. The image of an arrow is, by definition, the smallest subobject through which the arrow factors. A cover is a morphism whose image is the identical subobject. In a regular category, every arrow can be factored, in a unique way (up to a commuting isomorphism), as a cover followed by a monomorphism, and such factorisations are stable under pullback.

On a regular category $\cal C$ one can define a Grothendieck topology $J^{\textrm{reg}}_{{\cal C}}$ by stipulating that a sieve is $J^{\textrm{reg}}_{{\cal C}}$-covering if and only if it contains a cover. This topology is subcanonical, in other words $\cal C$ faithfully embeds into the associated topos $\Sh({\cal C}, J^{\textrm{reg}}_{{\cal C}})$. 

The natural notion of functor to consider between regular categories is that of regular functor: a functor between regular categories is said to be \emph{regular} if it preserves finite limits and covers. 

A regular category is said to be \emph{effective} (or \emph{Barr-exact}, cf. \cite{Barr}) if quotients by equivalence relations exist in it (and every equivalence relation is the kernel pair of such a quotient).

Every abelian category is effective regular. This fact plays an essential role in this paper.

\subsection{The exact completion of a regular category}\label{sec:exactcompletion}

Any regular category $\cal C$ can be fully faithfully embedded into an effective regular category ${\cal C}^{\textrm{eff}}$, called its \emph{effectivization}, characterized by the universal property that any regular functor from $\cal C$ to an effective regular category $\cal D$ extends uniquely to a regular functor ${\cal C}^{\textrm{eff}}\to {\cal D}$ (preserving coequalizers of equivalence relations).

By Remark D3.3.10 \cite{El}, we have an equivalence of toposes
\[
\Sh({\cal C}, J_{{\cal C}}^{\textrm{reg}})\simeq \Sh({\cal C}^{\textrm{eff}}, J_{{\cal C}^{\textrm{eff}}}^{\textrm{reg}}), 
\]
and the category ${\cal C}^{\textrm{eff}}$ can be recovered up to equivalence from the topos $\Sh({\cal C}, J_{{\cal C}}^{\textrm{reg}})$ as the full subcategory on its \emph{supercoherent objects}. A supercoherent object of a Grothendieck topos $\cal E$ is an object $A$ which is supercompact (in the sense that any covering of it in $\cal E$ contains a cover) and such that the domain of the kernel pair of any arrow to $A$ in $\cal E$ whose domain is a supercompact object is supercompact. 

As shown in \cite{Lack}, the effectivization ${\cal C}^{\textrm{eff}}$ of $\cal C$ is equivalent to the full subcategory of $\Sh({\cal C}^{\textrm{eff}}, J_{{\cal C}^{\textrm{eff}}}^{\textrm{reg}})$ on the objects which are coequalizers of equivalence relations in ${\cal C}$. 

The effectivization ${\cal C}^{\textrm{eff}}$ of a regular category $\cal C$ was explicitly described in \cite{CarboniVitale}, as follows. We shall denote by $SR$ the composition of two relations $R\mono X\times Y$ and $S\mono Y\times Z$, by $R^{o}\mono Y\times X$ the opposite of a relation $R\mono X\times Y$ and by $\leq$ the natural order relation between relations on the same pair of objects. The objects of ${\cal C}^{\textrm{eff}}$ are the pairs $(X, E)$, where $X$ is an object of $\cal C$ and $E$ is an equivalence relation on $X$, and the arrows $(X, E)\to (Y, F)$ are relations $R\mono X\times Y$ such that $RE=R=FR$ and $E\leq R^{o}R$ and $RR^{o}\leq F$. Let us denote by $q_{E}:X \to X\slash E$ and $q_{F}:Y \to Y\slash F$ the coequalizers of the equivalence relations $E$ and $F$ in ${\cal C}^{\textrm{eff}}$. Then, as observed in section 5 of \cite{Lack}, the arrow $\alpha_{R}:X\slash E \to Y\slash F$ induced by $R:(X, E) \to (Y, F)$ allows to reconstruct $R\mono X\times X$ by means of the following pullback square:
\[  
\xymatrix {
R \ar[d] \ar[r] &  Y  \ar[d]^{q_{F}} \\
X \ar[r]_{\alpha_{R} \circ q_{E}} & Y\slash F. } 
\] 

\begin{lemma}\label{lemma:conservativity}
Let $\cal C$ be a regular category and ${\cal C}^{\textrm{eff}}$ its effectivization. Let $F:{\cal C}\to {\cal D}$ be a regular functor to an effective regular category $\cal D$ and $\tilde{F}:{\cal C}^{\textrm{eff}} \to {\cal D}$ its extension to ${\cal C}^{\textrm{eff}}$. If $F$ is conservative then $\tilde{F}$ is conservative as well.
\end{lemma}   

\begin{proof}
To prove that $\tilde{F}$ is conservative if $F$ is, we observe that it suffices to show that for any monomorphism $m$, if $\tilde{F}(m)$ is an isomorphism then $m$ is an isomorphism. Indeed, this condition implies that $\tilde{F}$ is faithful and hence that $\tilde{F}$ reflects monomorphisms; but any arrow in a regular category can be factored as a cover followed by a monomorphism and if it is sent by an exact faithful functor to an isomorphism the cover part of it is monic (since its image by the functor is monic, it being an isomorphism), equivalently an isomorphism.   

Now, by Lemma 3.1 \cite{Lack}, the category $\cal C$ is closed under subobjects in ${\cal C}^{\textrm{eff}}$. By pulling back a monomorphism $m$ with codomain an object $X\slash E$ of ${\cal C}^{\textrm{eff}}$ along the cover $q_{E}:X\to X\slash E$ we thus obtain a subobject $n$ of $X$ in $\cal C$, and clearly $m$ is an isomorphism if and only if $n$ is (since covers are stable under pullback). From this our claim follows immediately.
\end{proof}

\subsection{Regular logic}\label{sec:regularlogic}

A \emph{first-order signature} $\Sigma$ consists of a set of \emph{sorts} (to be interpreted as \emph{sets} or more generally as \emph{objects} of a category), \emph{function symbols} (to be interpreted as \emph{functions} or more generally as \emph{arrows} of a category) and \emph{relation symbols} (to be interpreted as \emph{subsets} or more generally as \emph{subobjects} in a category). Constants are treated as $0$-ary function symbols.

For each sort $A$, one disposes of an infinite stock of variables $x^{A}$ of type $A$, and one can apply function symbols to them, and so on for a finite number of times, to form \emph{terms} over $\Sigma$. For example, the theory of commutative rings with unit has one sort, binary function symbols $+$, $-$ and $\cdot$ formalizing the operations on the ring and constants $0$ and $1$. The syntactic expression $(x\cdot y)+z$ is a term over this signature.

An \emph{atomic formula} over $\Sigma$ is a formula of the form $R(t_{1}, \ldots, t_{n})$, where $R$ is a $n$-ary relation symbol over $\Sigma$ and $t_{1}, \ldots, t_{n}$ is a $n$-tuple of terms over $\Sigma$.     

A first-order theory over a signature $\Sigma$ is said to be \emph{regular} if its axioms are of the form $(\phi \vdash_{\vec{x}} \psi)$, where $\phi$ and $\psi$ are regular formulae over $\Sigma$, i.e. formulae obtained from atomic formulae by only using finite conjunctions and existential quantifications. 

The regular syntactic category ${\cal C}^{\textrm{reg}}_{\mathbb T}$ of a regular theory $\mathbb T$ over a signature $\Sigma$ is the category having as objects the regular formulae-in-context $\{\vec{x}. \phi\}$ over $\Sigma$ (these are considered up to `renaming' equivalence) and as arrows $\{\vec{x}. \phi\}\to \{\vec{y}. \psi\}$ the $\mathbb T$-provable equivalence classes of regular formulae $\theta(\vec{x}, \vec{y})$ which are $\mathbb T$-provably functional from $\{\vec{x}. \phi\}$ to $\{\vec{y}. \psi\}$, i.e. such that the sequents
\[
(\theta \vdash_{\vec{x}, \vec{y}} \phi \wedge \psi),\\
(\phi \vdash_{\vec{x}} (\exists \vec{y})\theta), \textrm{ and }
(\theta(\vec{x}, \vec{y}) \wedge \theta(\vec{x}, \vec{y'}) \vdash_{\vec{x}, \vec{y}, \vec{y'}} \vec{y}=\vec{y'})
\]
are provable in $\mathbb T$ (where we suppose without loss of generality the contexts $\vec{x}$ and $\vec{y}$ to be disjoint). Notice that every tuple of terms $(t_{1}(\vec{x}), \ldots, t_{m}(\vec{x}))$ such that the sequent
$(\phi \vdash_{\vec{x}} \psi(t_{1}(\vec{x}), \ldots, t_{m}(\vec{x})))$ is provable in $\mathbb T$ defines an arrow in the syntactic category from $\{\vec{x}. \phi\}$ to $\{\vec{y}. \psi\}$. Anyway, in general not all the arrows of ${\cal C}^{\textrm{reg}}_{\mathbb T}$ are of this form (take for example the first-order (regular) theory of categories: this theory has a ternary relation symbol $C$ formalizing composition of arrows and the arrow $\{f,g. \dom(g)=\cod(g)\}\to \{h. \top\}$ given by $[C(h,f,g)]$ is not provably equivalent to a term). It should be noted that the condition for a formula to be provably functional corresponds, semantically, to the requirement that it is the graph of a morphism from the interpretation of the formula in the domain to the interpretation of the formula in the codomain. If one wants ${\cal C}^{\textrm{reg}}_{\mathbb T}$ to be a regular category, these are the arrows that one has to take; terms do not suffice in general. Indeed, given a $\mathbb T$-provably functional formula $\theta(\vec{x}, \vec{y})$ from $\{\vec{x}. \phi\}$ to $\{\vec{y}. \psi\}$, the canonical projection arrow $[\vec{x'}=\vec{x}]:\{\vec{x}, \vec{y}. \theta\}\to \{\vec{x'}. \phi(\vec{x'}\slash \vec{x})\}$ is a cover (in a regular category existential quantifications are interpreted by taking images) and a monomorphism (by the second of the functionality axioms), whence it should be an isomorphism (as in a regular category every arrow which is both a cover and a monomorphism is an isomorphism). In particular, there should be an arrow $\{\vec{x'}. \phi(\vec{x'}\slash \vec{x})\}\to \{\vec{x}, \vec{y}. \theta\}$ which is inverse to $[\vec{x'}=\vec{x}]$, i.e. which is given by $[\theta(\vec{x}, \vec{y}) \wedge \vec{x'}=\vec{x}]$. 

\begin{theorem}\protect{\rm (\cite{MR} and \cite{El})}
The regular syntactic category ${\cal C}^{\textrm{reg}}_{\mathbb T}$ satisfies the following universal property: it is a regular category and for any regular category $\cal D$, there is an equivalence of categories
\[
{\mathbb T}\textrm{-mod}({\cal D})\simeq \textbf{Reg}({\cal C}^{\textrm{reg}}_{\mathbb T}, {\cal D}),
\]
natural in $\cal D$, where ${\mathbb T}\textrm{-mod}({\cal D})$ denotes the category of $\mathbb T$-models in $\cal D$ and structure homomorphisms between them and $\textbf{Reg}({\cal C}^{\textrm{reg}}_{\mathbb T}, {\cal D})$ denotes the category of regular functors ${\cal C}^{\textrm{reg}}_{\mathbb T} \to {\cal D}$ and natural transformations between them. 
\end{theorem}

One half of the equivalence of the theorem sends a model $M$ to the functor $F_{M}$ sending $\{\vec{x}. \phi\}$ to its interpretation $[[\vec{x}. \phi]]_{M}$ in $M$ and acting accordingly on the arrows. In particular, for any $\mathbb T$-provably functional formula $\theta(\vec{x}, \vec{y}):\{\vec{x}. \phi\}\to \{\vec{y}. \psi\}$ and any $\mathbb T$-model homomorphism $f:M\to N$, the following diagram is commutative:

\[  
\xymatrix {
[[\vec{x}. \phi]]_{M} \ar[d]^{f} \ar[r]^{[[\theta]]_{M}} &  [[\vec{y}. \psi]]_{M} \ar[d]^{f}  \\
[[\vec{x}. \phi]]_{N} \ar[r]^{[[\theta]]_{N}}  & [[\vec{y}. \psi]]_{N}.}
\] 

For more details, see section D1.4 \cite{El} or \cite{OC}. 

A model $M$ of a regular theory $\mathbb T$ is said to be \emph{conservative} if every regular sequent over the signature of $\mathbb T$ which is valid in $M$ is provable in $\mathbb T$. This terminology is justified by the fact that, for any $M$, $F_{M}$ is conservative (as a functor) if and only if $M$ is conservative (as a $\mathbb T$-model). Notice that, since ${\cal C}^{\textrm{reg}}_{\mathbb T}$ has equalizers and $F_{M}$ preserves them, if $M$ is conservative $F_{M}$ is also faithful. 

The classifying topos of a regular theory $\mathbb T$ can be constructed as the topos of sheaves $\Sh({\cal C}^{\textrm{reg}}_{\mathbb T}, J_{\mathbb T}^{\textrm{reg}})$ on the regular syntactic category of $\mathbb T$ with respect to the regular topology on it, or as the topos of sheaves on the effectivization of ${\cal C}^{\textrm{reg}}_{\mathbb T}$ with respect to the regular topology on it. 

The effectivization of ${\cal C}^{\textrm{reg}}_{\mathbb T}$ can be recovered from the classifying topos ${\cal E}_{\mathbb T}$ of $\mathbb T$ as the full subcategory on the supercoherent objects. As observed above, it can also be characterized as the closure of ${\cal C}^{\textrm{reg}}_{\mathbb T}$ in ${\cal E}_{\mathbb T}$ under quotients by equivalence relations.

\section{Nori motives}

Now that we have provided the necessary background, we can proceed with our logical analysis of Nori's construction.

\subsection{Review of Nori's construction}\label{sec:review}

We denote by $\Rmod$ the category of $R$-modules for a commutative ring with unit $R$, and by $ \Rmod_{\rm f}$ the full subcategory of $\Rmod$ on the finitely generated $R$-modules. If $R$ is a field $k$, we shall denote by $k\textrm{-vect}_{\textrm{f}}$ the category of finite-dimensional vector spaces over $k$. Given a category with finite products $\cal A$, we shall denote by $\Rmod({\cal A})$ the category of $R$-modules internal to $\cal A$, that is the category of models of the algebraic theory of $R$-modules in the category $\cal A$. 

Let $D$ be a \emph{diagram} (i.e., an oriented graph) and let $T:D \to \Rmod$ be a \emph{representation} of $D$ into $R$-modules (i.e., a map sending each vertex $d$ of $D$ to an abelian group $T(d)$ in $\Rmod$ and each edge $f:c\to d$ of $D$ to a $R$-linear homomorphism $T(c)\to T(d)$). 

Given a representation $T:D \to \Rmod$, we denote by $\End(T)$ the set of endomorphisms of $T$, meaning functions $\alpha$ which assign to each vertex $d$ of $D$ an $R$-linear homomorphism $\alpha(d):T(d)\to T(d)$ in such a way that for every edge $f:c\to d$ of $D$, the following diagram commutes:
\[  
\xymatrix {
T(c) \ar[d]^{T(f)} \ar[r]^{\alpha(c)} & T(c) \ar[d]^{T(f)} \\
T(d)  \ar[r]^{\alpha(d)} & T(d). } 
\] 

For a diagram $D$ and a representation $T:D \to \Rmod_{\rm f}$, we denote by $\End(T)\textrm{-mod}_{\textrm{fin}}$ the category of $\End(T)$-modules which are finitely generated as $R$-modules and homomorphisms between them.

Recall the following key result due to Nori (proofs of it are given in \cite{Brug} and \cite{NM}).

\begin{theorem}\protect{\rm(Nori)}\label{thm:Nori}
Let $D$ be a diagram and $T:D\to \Rmod_{\rm f}$ a representation of $D$, where $R$ is a Noetherian ring. Then there exists an $R$-linear abelian category ${\cal C}_{T}$, a representation $\tilde{T}:D \to {\cal C}_{T}$ and an exact faithful functor $F_{T}:{\cal C}_{T}\to \Rmod_{\rm f}$ such that $T=F_{T}\circ \tilde{T}$ and this factorisation is universal in the sense that for any factorisation $T=F\circ S$, where $F:{\cal A}\to \Rmod_{\rm f}$ is an exact and faithful functor defined on an $R$-linear abelian category $\cal A$ and $S:D \to {\cal A}$ is a representation of $D$ in $\cal A$, there exists a unique, up to isomorphism, exact (and faithful) functor $\xi:{\cal C}_{T} \to {\cal A}$ such that the following diagram commutes (up to isomorphism):

\begin{center}
\begin{tikzpicture}
\node (1) at (2,0) {${\cal C}_{T}$};
\node (4) at (0,-3) {$D$};
\node (5) at (4,-3) {$\Rmod_{\rm f}$};
\node (6) at (2,-1.5) {$\cal A$};
\draw[->](4) to node [above]{$T$} (5);
\draw[->](4) to node [above]{$S$} (6);
\draw[->](6) to node [above]{$F$} (5);
\draw[->](1) to node [right]{$F_{T}$} (5);
\draw[dashed, ->](1) to node [right]{$\xi$} (6);
\draw[->](4) to node [left] {$\tilde{T}$} (1);
\end{tikzpicture}
\end{center}

\end{theorem}

\begin{remark}\label{rem:isomorphisms}
A clarification on the statement of the theorem is in order (see \cite{Brug}). The uniqueness up to isomorphism of the functor $\xi$ means the following: there are isomorphisms $\alpha:\xi\circ \tilde{T}\longby{\sim} S$ and $\beta:F\circ \xi\longby{\sim} F_{T}$ such that $F(\alpha)=\beta \tilde{T}$, and if $(\xi', \alpha', \beta')$ is another solution to the factorisation problem, there exists an isomorphism $\gamma:\xi'\longby{\sim} \xi$ such that $\alpha'=(\gamma \tilde{T})\alpha$ and $\beta'=\beta(F \gamma)$. 

\end{remark}
 
Nori's point of view in constructing ``mixed motives'' consists in thinking of an ``homology theory'' as a representation $T$ of Nori's diagram described in the introduction. In particular, $T$ can be provided by singular homology, and Nori's category of ``effective homological motives'' is given by ${\cal C}_{T}$ for this particular representation (see \cite{LV}, \cite{NN} and \cite{HMS}). 
 
The construction $\cC_T$ satisfies the following key properties, which in fact suffice to derive the universal property of Theorem \ref{thm:Nori}:
\begin{itemize}
\item If $D$ is a finite diagram then ${\cal C}_{T}\simeq \End(T)\text{-{\rm mod}}_{\rm fin}$;

\item if we have a map of diagrams $\iota: D^\prime \to D$ such that $T^\prime = T\circ \iota$, there is a canonical functor $$\iota_*:\cC_{T^\prime}\to \cC_T$$ making the diagram
$$\xymatrix{D^\prime\ar[r]^-{\iota}  \ar[d]^-{\tilde T^\prime} \ar@/_4.5pc/[ddr]_-{T^\prime} & D\ar[d]_-{\tilde T}\ar@/^2.5pc/[dd]^-{T}\\
\cC_{T^\prime}\ar[dr]_-{F_{T^\prime}} \ar@{.>}[r]^{\iota_*} & \cC_T\ar[d]^-{F_{T}}\\
&\Rmod_{\rm f} }$$
commutative (up to isomorphism);
\item  if $\cA$ is an abelian $R$-linear category and $T : \cA \to \Rmod_{\rm f}$ is a faithful, exact $R$-linear functor then $\tilde{T}: \cA \by{\simeq}\cC_T$ is an equivalence.
\end{itemize}

In fact, for {\it finite} (full) subdiagrams $\iota : F^\prime\into F$ of $D$ we have $$\iota_* : \cC_{T|_{F^\prime}} = \End (T|_{F^\prime})\text{-{\rm mod}}_{\rm fin}\to \cC_{T|_{F}}= \End (T|_{F})\text{-{\rm mod}}_{\rm fin}$$ so we can define
$$\cC_T= \colim{F \subseteq D \textrm{ finite}}\End (T|_{F})\text{-{\rm mod}}_{\rm fin}$$
taking the (filtered) colimit over all such subdiagrams. Furthermore, the diagram
$$\xymatrix{D\ar[r]^-{S}  \ar[d]^-{\tilde T} \ar@/_4.5pc/[ddr]_-{T} & \cA\ar@/^2.5pc/[dd]^-{F}\\
\cC_T\ar@/_0.5pc/[ur]^-{\xi}\ar[dr]_-{F_{T}} \ar@{.>}[r] _-{S_*} & \cC_{F}\ar@{.>}[u]_-{}\ar[d]^-{F_{F}}\\
&\Rmod_{\rm f} }$$
commutes, where $\xi$ is defined as the composition of the dotted arrows.

If $R$ is a field $k$, there is a more compact description of Nori's category which enlightens its relationship with the Tannakian formalism, we need to recall the definition of the coalgebra $\End^{\vee}(T)$ of endomorphisms of a representation $T:D\to k\textrm{-vect}_{\textrm{f}}$.

The algebra $\End(T)$ of endomorphisms of the functor $T$ can be identified with the kernel of the map 
\[
S:\mathbin{\mathop{\textrm{ $\prod$}}\limits_{d\in D}} \Hom_{k}(T(d), T(d)) \to \mathbin{\mathop{\textrm{ $\prod$}}\limits_{f:d'\to d'' \textrm{ in } D}} \Hom_{k}(T(d'), T(d''))
\]   
sending an element $<\alpha_{d}:T(d)\to T(d) \mid d\in D>$ of $\mathbin{\mathop{\textrm{ $\prod$}}\limits_{d\in D}} \Hom_{k}(T(d), T(d))$ to the element $<T(f)\circ \alpha_{d'}-\alpha_{d''}\circ T(f) \mid f:d'\to d'' \textrm{ in } D>$ of $\mathbin{\mathop{\textrm{ $\prod$}}\limits_{f:d'\to d'' \textrm{ in } D}} \Hom_{k}(T(d'), T(d''))$.

The coalgebra $\End^{\vee}(T)$ is defined in such a way that   $\End^{\vee}(T)^{\ast}\cong \End(T)$: it is set equal to the cokernel of the map
\[
\mathbin{\mathop{\textrm{ $\bigoplus$}}\limits_{f:d'\to d'' \textrm{ in } D}} \Hom_{k}(T(d'), T(d''))^{\ast} \to \mathbin{\mathop{\textrm{ $\bigoplus$}}\limits_{d\in D}} \Hom_{k}(T(d), T(d))^{\ast}
\]
defined in such a way that its dual is the map $S$ defined above. For more details about this construction we refer the reader to \cite{JS} and \cite{Arapura}.

It is important to note that $\End^{\vee}(T)\ncong \End(T)^{\ast}$ in general, although this is clearly true if $D$ is finite. 

As observed in \cite{NM} (see also \cite{Arapura}), we have canonical equivalences
\[
\End(T|_{F})\textrm{-mod}_{\textrm{fin}}\simeq \textrm{Comod}_{\textrm{fin}}(\End^{\vee}(T|_{F}))) 
\]
and
\[
\textrm{Comod}_{\textrm{fin}}(\End^{\vee}(T))=\colim{F\subseteq D \textrm{ finite }}\textrm{Comod}_{\textrm{fin}}(\End^{\vee}(T|_{F})),
\]
yielding an alternative, elegant, description of Nori's category of $T$ as the category $\textrm{Comod}_{\textrm{fin}}(\End^{\vee}(T)$ of finite-dimensional comodules over the coalgebra $\End^{\vee}(T)$. This relies on the fact that for any filtered colimit $A=\colim{i\in {\cal I}}A_{i}$ of coalgebras, $\textrm{Comod}_{\textrm{fin}}(A)=\colim{i\in {\cal I}}\textrm{Comod}_{\textrm{fin}}(A_{i})$ and on the fact that for any finite-dimensional coalgebra $A$ there is a canonical equivalence $$\textrm{Comod}_{\textrm{fin}}(A^{\vee})\simeq A\textrm{-mod}_{\textrm{fin}}$$ between the category of left $A^{\vee}$-comodules of finite dimension and that of left $A$-modules of finite dimension (cf. p. 73 of \cite{NM}); indeed, we clearly have $\End^{\vee}(T)=\mathbin{\mathop{\textrm{ $\bigcup$}}\limits_{F\subseteq D \textrm{ finite }}}\End^{\vee}(T|_{F})$.

\subsection{The main theorem}

Let $D$ be a diagram and $T:D\to \Rmod$ be a representation of $T$, where $R$ is a ring.

Let us define the regular theory ${\mathbb T}_{T}$ of $T$ as follows. The signature $L_{D}$ of ${\mathbb T}_{T}$ has one sort for each object $d$ of $D$, one function symbol for each arrow of $D$, and constants and function symbols formalizing the structure of left $R$-module on each sort $d$ (to formalize scalar multiplication we take, for each sort $c$, one unary function symbol of sort $c$ for each element of $R$). Notice that the language $L_{D}$ depends on the ring $R$ of coefficients but not on $T$. 

\begin{definition}
The theory ${\mathbb T}_{T}$ is the regular theory $\Th(T)$ of $T$ as a $L_{D}$-structure, that is the set of regular sequents over $L_{D}$ which are satisfied in $T$. 
\end{definition}

Notice in particular that the axioms of $R$-linearity of the interpretations of the arrows in $D$ are provable in ${\mathbb T}_{T}$ (since they are valid in $T$).

Let ${\cal C}_{{\mathbb T}_{T}}$ be the effectivization of the regular syntactic category of ${\mathbb T}_{T}$. We clearly have a representation $\tilde{T}:D\to {\cal C}_{{\mathbb T}_{T}}$ given by:
\[
d \leadsto  \{x^{d}. \top\}
\]
and 
\[
(f:d\to d') \leadsto  ([f]:\{x^{d}. \top\} \to \{x^{d'}. \top\}).\]

\begin{lemma}\label{lemma:additive}
The category ${\cal C}_{{\mathbb T}_{T}}$ is additive and $R$-linear. 
\end{lemma}

\begin{proof}
Due to the definition of the theory ${\mathbb T}_{D}$, the objects $\{x^{d}. \top \}$ of ${\cal C}_{{\mathbb T}_{T}}$ have the structure of internal (left) $R$-modules in ${\cal C}_{{\mathbb T}_{T}}$ (by an internal $R$-module we mean a model of the theory of $R$-modules, axiomatized over a signature having a function symbol for each element of $R$, the constant $0$ and the addition symbol formalizing the group operation). These structures naturally induce a structure of internal $R$-module on each object of ${\cal C}_{{\mathbb T}_{T}}$, as follows. By definition of ${\cal C}_{{\mathbb T}_{T}}$, every object of ${\cal C}_{{\mathbb T}_{T}}$ is obtained by repeatedly applying finite limits, images and quotients to objects of the form $\{x^{d}. \top \}$ and arrows between them of the form $[t]:\{x^{d}. \top \} \to \{x^{d'}. \top \}$, where $t$ is a term over $L_{D}$. Now, finite limits of $R$-modules, images under $R$-linear maps between $R$-modules and quotients of $R$-modules by $R$-submodules are canonically endowed with the structure of an $R$-submodule; internalizing this remark to ${\cal C}_{{\mathbb T}_{T}}$ proves our claim. In particular, for any regular formula $\{\vec{x}. \phi\}$ over $L_{D}$, the structure of internal $R$-module on $\{\vec{x}. \phi\}$ is the restriction (in the sense of the commutativity of an obvious diagram) of the structure of internal $R$-module on $\{\vec{x}. \top\}$. This implies that any ${\mathbb T}_{T}$-provably functional formula $\{\vec{x}. \phi\}\to \{\vec{y}. \psi\}$ is internally a $R$-linear map of (internal) $R$-modules (since its graph is a $R$-submodule of $\{\vec{x}, \vec{y}. \phi \wedge \psi\}$  by this remark). This clearly extends to the objects and arrows of the effectivization ${\cal C}_{{\mathbb T}_{T}}$. It follows that ${\cal C}_{{\mathbb T}_{T}}$ is a pre-additive category, i.e. its hom-sets are endowed with the structure of an $R$-module and the composition maps are bi-linear. Indeed, given objects $a$ and $b$ of ${\cal C}_{{\mathbb T}_{T}}$, the operations of $R$-module on $b$ can be used to define the sum and $R$-scalar product for arrows from $a$ to $b$. In particular, ${\cal C}_{{\mathbb T}_{T}}$ is a $R$-linear category.

Next, we notice that ${\cal C}_{{\mathbb T}_{T}}$ has a zero object, that is the terminal object $\{[]. \top\}$ is also initial in ${\cal C}_{{\mathbb T}_{T}}$. The fact that it is weakly initial, i.e. that there is at least an arrow from it to any object in the effectivization of ${\cal C}_{{\mathbb T}_{T}}$, follows from the just remarked fact that every object of ${\cal C}_{{\mathbb T}_{T}}$ has the structure of an internal $R$-module (since the constant $0$ defines a provably functional formula from $\{[]. \top\}$ to it). It remains to show that any two arrows from $\{[]. \top\}$ to an object of ${\cal C}_{{\mathbb T}_{T}}$ are equal. To this end, consider their equalizer $m:E \mono \{[]. \top\}$. Since the object $\{[]. \top\}$ is weakly initial, there is an arrow $\alpha:\{[]. \top\} \to E$. On the other hand, $\{[]. \top\}$ being terminal, there is an arrow $\beta:E \to \{[]. \top\}$, whose composite with $\alpha$ is the identity, and which is equal to $m$. Therefore $\beta$ is split epic; since it is also monic, it is an isomorphism, as required.   

Since ${\cal C}_{{\mathbb T}_{T}}$ has finite limits, a zero object and is pre-additive, it has finite biproducts. Hence it is an additive category.
\end{proof}

\begin{theorem}\label{thm:main}
Let $R$ be a ring, $D$ a diagram and $T:D\to \Rmod$ a representation. Then the category ${\cal C}_{{\mathbb T}_{T}}$ is abelian and $R$-linear and, together with the representation $\tilde{T}$ and the functor $F_{T}$ defined above, satisfies the following universal property: $T=F_{T}\circ \tilde{T}$ and this factorisation is universal in the sense that for any factorisation $T=F\circ S$, where $F:{\cal A}\to \Rmod$ is an exact and faithful functor defined on an $R$-linear abelian category $\cal A$ and $S:D \to {\cal A}$ is a representation of $D$ in $\cal A$, there exists a unique exact (and faithful) functor $F_{S}:{\cal C}_{T} \to {\cal A}$ such that the following diagram commutes:

\begin{center}
\begin{tikzpicture}
\node (1) at (2,0) {${\cal C}_{T}$};
\node (4) at (0,-3) {$D$};
\node (5) at (4,-3) {$\Rmod$};
\node (6) at (2,-1.5) {$\cal A$};
\draw[->](4) to node [above]{$T$} (5);
\draw[->](4) to node [above]{$S$} (6);
\draw[->](6) to node [above]{$F$} (5);
\draw[->](1) to node [right]{$F_{T}$} (5);
\draw[dashed, ->](1) to node [right]{$F_{S}$} (6);
\draw[->](4) to node [left] {$\tilde{T}$} (1);
\end{tikzpicture}
\end{center}

If $R$ is Noetherian and $T$ takes values in $\Rmod_{\rm f}$ then $F_{T}$ takes values in $\Rmod_{\rm f}$ as well. 
\end{theorem} 

\begin{proof}
From Lemma \ref{lemma:additive} we know that the category ${\cal C}_{{\mathbb T}_{T}}$ is additive and $R$-linear. By a general (unpublished) theorem of Tierney, whose proof can for instance be found in \cite{Barr}, every additive exact category is abelian. So ${\cal C}_{{\mathbb T}_{T}}$ is abelian. This can also be directly proved as follows. Since ${\cal C}_{{\mathbb T}_{T}}$ is additive and $R$-linear, to prove that it is abelian it remains to show that
\begin{enumerate}[(1)]
\item every morphism has a kernel and a cokernel, and
\item every monomorphism and every epimorphism is normal (that is, every monomorphism is a kernel of some morphism, and every epimorphism is a cokernel of some morphism). 
\end{enumerate}

(1) The fact that every morphism has a kernel follows from the fact that the kernel of a morphism $f$ can be described as the equalizer of $f$ and the zero arrow and that ${\cal C}_{{\mathbb T}_{T}}$ has all finite limits, in particular equalizers. The fact that every morphism $f$ has a cokernel follows from the fact that, ${\cal C}_{{\mathbb T}_{T}}$ being effective, quotients by equivalence relations exist, and the equivalence relation $R_{f}$ given (in the internal language) by: `$(x, y)\in R_{f}$ if and only if $x-y \in \im(f)$' is such that the quotient of the codomain of $f$ by $R_{f}$ is precisely the cokernel of $f$. 

(2) Let us prove that every monomorphism is the kernel of its cokernel pair. This follows from the fact that in every topos a monomorphism is the equalizer of its cokernel pair, in light of the fact that the embedding of ${\cal C}_{{\mathbb T}_{T}}$ into the classifying topos of ${\mathbb T}_{T}$ preserves finite limits and coequalizers of equivalence relations in ${\cal C}_{{\mathbb T}_{T}}$ (see section \ref{sec:preliminaries}). Indeed, as we observed above, all the arrows in ${\cal C}_{{\mathbb T}_{T}}$ are (internally) $R$-linear, whence the kernel pair of the cokernel pair of a monomorphism is isomorphic to the kernel pair of the coequalizer of the monomorphism and the zero arrow to its codomain.

It remains to prove that every epimorphism $f$ is the cokernel of its own kernel. Now, in every regular category a cover is the coequalizer of its kernel pair. By factoring $f$ as a cover $e$ followed by a monomorphism $m$, we obtain that $m$ is an epimorphism and hence, it being regular, an isomorphism. Therefore $f$ is a cover and hence the coequalizer of its kernel pair. Now, since $f$ is (internally) $R$-linear, the coequalizer of its kernel pair coincides with the cokernel of its kernel, whence $f$ is the cokernel of its kernel, as required.   

Since $T$ is a model in $\Set$ of the theory ${\mathbb T}_{T}$, the universal property of the regular syntactic category of ${\mathbb T}_{T}$ and of its effectivization yields an exact functor $F_{T}:{\cal C}_{{\mathbb T}_{T}} \to \Set$. The fact that $T$ is a conservative ${\mathbb T}_{T}$-model implies, by Lemma \ref{lemma:conservativity},  that $F_{T}$ is conservative and hence faithful. The functor $F_{T}$ actually takes values in $\Rmod$ and is exact with values in this category. Indeed, the category $\Rmod$ of $R$-modules being monadic over $\Set$, the forgetful functor $\Rmod\to \Set$ preserves and reflects finite limits and coequalizers of equivalence relations (cf. Proposition 3.5.2 \cite{borceux}). Notice that if $R$ is Noetherian the subcategory $\Rmod_{\textrm{f}}$ is closed in $\Rmod$ under finite limits and coequalizers of equivalence relations. Moreover, any regular formula $\{\vec{x}. \phi\}$ over $L_{D}$ is sent by $F_{T}$ to its interpretation $[[\vec{x}. \phi]]_{T}$ in the model $T$, which is a subobject of the interpretation of $\{\vec{x}. \top\}$ in $T$, namely of $Td_{1}\times \cdots \times Td_{n}$ (if $\vec{x}=(x_{1}^{d_{1}}, \ldots, x_{n}^{d_{n}})$), whence, if $R$ is Noetherian and $T$ takes values in $\Rmod_{\rm f}$, $[[\vec{x}. \phi]]_{T}$ lies in $\Rmod_{\rm f}$ as well.  

We have already proved that ${\cal C}_{{\mathbb T}_{T}}$ is an $R$-linear abelian category, that $\tilde{T}$ is a representation of $D$ in it and that $F_{T}$ is a faithful exact functor ${\cal C}_{{\mathbb T}_{T}} \to \Rmod$. We clearly have that $T=F_{T}\circ \tilde{T}$. Suppose that $T=F\circ S$ is a factorisation of $T$ through an $R$-linear abelian category $\cal A$ with a representation $S:D\to {\cal A}$ and a faithful exact functor $F:{\cal A}\to \Rmod$. The representation $S$ defines a model of the theory ${\mathbb T}_{T}$ inside the category $\cal A$; indeed, it is clearly a $L_{D}$-structure and it is a model of ${\mathbb T}_{T}$ since the functor $F$ is exact and faithful whence it reflects the validity of regular sequents over $L_{D}$. Since $\cal A$ is effective regular (it being abelian), the universal property of ${\cal C}_{{\mathbb T}_{T}}$ yields a unique exact functor $F_{S}:{\cal C}_{{\mathbb T}_{T}}\to {\cal A}$ sending the model $\tilde{T}$ to the model $S$, i.e. such that $F_{S}\circ \tilde{T}=S$. It remains to prove that $F\circ F_{S}=F_{T}$. Since both $F\circ F_{S}$ and $F_{T}$ are regular functors defined on ${\cal C}_{{\mathbb T}_{T}}$, to prove that they are equal is equivalent to show that the models of ${\mathbb T}_{T}$ corresponding to them are the same; now, the model of ${\mathbb T}_{T}$ corresponding to $F_{T}$ is $T$, while the one corresponding to $F\circ F_{S}$ is $F \circ S$; but $T=F\circ S$, as required. This completes the proof of the theorem.      
\end{proof}

\begin{remarks}\label{rem:generalization}
\begin{enumerate}[(a)]
\item The syntactic characterization of Nori's category provided by Theorem \ref{thm:main} yields a stronger universal property with respect to the one considered by Nori, directly arising from the universal property of the effectivization of the regular syntactic category of a regular theory recalled in section \ref{sec:preliminaries}:  for any effective regular category $\cal A$ with a representation $S:D\to \Rmod({\cal A})$ satisfying exactly the same properties (expressible in regular logic over $L_{D}$) as $T$, there exists a unique exact functor $F_{S}:{\cal C}_{{\mathbb T}_{T}} \to \Rmod({\cal A})$ such that $F_{S} \circ \tilde{T}=S$: 
\begin{center}
\begin{tikzpicture}
\node (1) at (3,0) {${\cal C}_{{\mathbb T}_{T}}$};
\node (6) at (3,-2) {$\Rmod({\cal A})$};
\node (2) at (0, -2) {$D$};
\draw[->](2) to node [above]{$\tilde{T}$} (1);
\draw[->](2) to node [above]{$S$} (6);
\draw[dashed, ->](1) to node [right]{$F_{S}$} (6);
\end{tikzpicture}
\end{center}

\item The methodology that we have exploited to build Nori categories is very general and can be adapted to construct other structures satisfying similar universal properties. For instance, the effectivization of the regular syntactic category of the theory over $L_{D}$ (taking $R={\mathbb Z}$) containing just the axioms formalizing the structure of abelian group on each sort $d$ satisfies the universal property of the free abelian category on the diagram $D$ (notice that any representation of a diagram in an abelian category $\cal A$ actually takes values in the category of abelian groups internal to $\cal A$).  

\item Any abelian category $\cal A$ can be fully faithfully embedded in a Grothendieck topos, namely the topos of regular sheaves on it, and recovered from it as the full subcategory on its supercoherent objects. The universal property of Theorem \ref{thm:main} can thus be viewed as arising from that of the classifying topos ${\cal E}_{{\mathbb T}_{T}}$ of ${\mathbb T}_{T}$; indeed, the functor $F_{S}:{\cal C}_{{\mathbb T}_{T}}\to {\cal A}$ is the restriction to the full subcategories of supercoherent objects of the inverse image functor of the geometric morphism $\Sh({\cal A}, J_{{\cal A}}^{\textrm{reg}})\to {\cal E}_{{\mathbb T}_{T}}\simeq \Sh(  {\cal C}_{{\mathbb T}_{T}}, J_{{\cal C}_{{\mathbb T}_{T}}}^{\textrm{reg}})$ induced by the model $y_{\cal A}\circ S$ of ${\mathbb T}_{T}$ in $\Sh({\cal A}, J_{{\cal A}}^{\textrm{reg}})$, where $y_{{\cal A}}$ is the Yoneda embedding ${\cal A}\hookrightarrow \Sh({\cal A}, J_{{\cal A}}^{\textrm{reg}})$.

\item The syntactic category ${\cal C}_{{\mathbb T}_{T}}$ satisfies a stronger universal property than that of Theorem \ref{thm:Nori}; that is, the uniqueness of the functor $F_{S}$ making the left-hand triangle (strictly) commute is strict and not up to isomorphism as in Remark \ref{rem:isomorphisms}.      
\end{enumerate}
\end{remarks}

\subsection{Generating structured subcategories}\label{sec:subcategories}

A natural problem, which is particularly relevant for the purposes of this paper, is that of explicitly describing the regular (resp. effective regular, abelian) subcategory ${\cal C}_{\cal F}$ of a regular (resp. effective regular, abelian) category $\cal C$ generated by a given family $\cal F$ of objects and arrows in $\cal C$, in the sense of being the smallest regular (resp. effective regular, abelian) subcategory of $\cal C$ containing $\cal F$. 

Such a category always exist, since the property of being closed under finite limits and images (resp., finite limits, images and coequalizers of equivalence relations,  finite products, kernels and cokernels) is stable under intersection of subcategories (recall that our categories are endowed with \emph{canonical} choices of the relevant limits and colimits, so that we have well-defined limit or colimit functors, for diagrams of the appropriate shapes, on them). Such a subcategory can thus be built by means of an inductive process, as follows. We set ${\cal C}^{0}_{\cal F}$ equal to the subcategory of $\cal C$ generated by the objects and arrows in $\cal F$. For any natural number $n\in {\mathbb N}$, we set ${\cal C}^{n+1}_{\cal F}$ equal to the subcategory of $\cal C$ generated by the objects and arrows obtained by applying the finite limit and image functors (resp., the finite limits, images and coequalizers of equivalence relations functors, the finite product, kernel and cokernel functors) to the diagrams with values in ${\cal C}^{n}_{\cal F}$. It is clear that the union ${\cal C}_{\cal F}$ of the subcategories ${\cal C}^{n}_{\cal F}$ (for $n\in {\mathbb N}$) is the smallest regular (resp. effective regular, abelian) subcategory ${\cal C}_{\cal F}$ of $\cal C$ containing $\cal F$.     
 
The notion of syntactic category recalled in section \ref{sec:preliminaries} comes to our aid in obtaining more explicit descriptions of the categories ${\cal C}_{\cal F}$, as follows. 

Given a regular category $\cal C$, we can attach to the family $\cal F$ a signature $\Sigma_{\cal F}$ having a sort for each object in $\cal F$ and a function symbol for each arrow in $\cal F$. The embedding ${\cal F}\hookrightarrow {\cal C}$ defines a $\Sigma_{\cal F}$-structure $T_{{\cal F}}$ in $\cal C$. We can consider the regular theory $\Th(T_{{\cal F}})$ of this structure over $\Sigma_{\cal F}$. By definition of $\Th(T_{{\cal F}})$, the structure $T_{{\cal F}}$ is a conservative $\Th(T_{{\cal F}})$-model in $\cal C$. Therefore the embedding ${\cal F}\hookrightarrow {\cal C}$ extends to a conservative (faithful) exact functor $F_{\cal F}$ from the regular syntactic category ${\cal C}^{\rm reg}_{\Th(T_{{\cal F}})} $ of the theory $\Th(T_{{\cal F}})$ to ${\cal C}$. If $\cal C$ is effective regular, the functor $F_{\cal F}$ extends, by Lemma \ref{lemma:conservativity}, to a conservative (faithful) exact functor $F_{\cal F}$ from ${{\cal C}^{\rm reg}_{\Th(T_{{\cal F}})}}^{\rm eff}$ to $\cal C$.  

As shown by the following theorem, if the functor $F_{\cal F}$ is full on objects then it yields an equivalence between ${\cal C}^{\rm reg}_{\Th(T_{{\cal F}})}$ (resp. ${{\cal C}^{\rm reg}_{\Th(T_{{\cal F}})}}^{\rm eff}$  ) and the regular (resp. effective regular) subcategory ${\cal C}_{\cal F}$ of $\cal C$ generated by $\cal F$. 

\begin{theorem}\label{thm:regsubcategory}
Let $\cal C$ be a small category and $\cal F$ a family of objects and arrows of $\cal C$. Then

\begin{enumerate}[(i)]
\item If $\cal C$ is regular then the functor $F_{\cal F}:{\cal C}^{\rm reg}_{\Th(T_{{\cal F}})} \to {\cal C}$ yields an equivalence of categories into the regular subcategory ${\cal C}_{\cal F}$ of $\cal C$ generated by $\cal F$ if and only if it is full on objects;

\item If $\cal C$ is effective regular then the functor $F_{\cal F}:{{\cal C}^{\rm reg}_{\Th(T_{{\cal F}})}}^{\rm eff} \to {\cal C}$ yields an equivalence of categories into the effective regular subcategory ${\cal C}_{\cal F}$ of $\cal C$ generated by $\cal F$ if and only if it is full on objects.  
\end{enumerate}
\end{theorem}  

\begin{proof}
The fullness on objects is clearly a necessary condition for the functor to yield an equivalence of categories, so it remains to prove that it is a sufficient condition. Since $F_{\cal F}$ is faithful, by Lemma \ref{lemma:fullonobjects} it yields an equivalence onto its image. So, what we have to prove is that the category ${\cal C}_{\cal F}$ equals the image of $F_{\cal F}$.      

If $\cal C$ is regular, the functor $F_{\cal F}:{\cal C}^{\rm reg}_{\Th(T_{{\cal F}})} \to {\cal C}$ sends each regular formula-in-context $\{\vec{x}. \phi\}$ over $\Sigma_{\cal F}$ to its interpretation in the structure $T_{{\cal F}}$. Now, any object $\{\vec{x}. \phi\}$ of ${\cal C}^{\rm reg}_{\Th(T_{{\cal F}})}$ is canonically obtained by taking finite limits and images starting from objects of the form $\{x^{c}. \top\}$ (for an object $c$ of $\cal F$) and arrows between them which are given by terms over the signature $\Sigma_{\cal F}$. Since $F_{\cal F}$ is exact and the category ${\cal C}_{\cal F}$ contains $\cal F$ and is closed under finite limits and images in $\cal C$, it follows that the image of $F_{\cal F}$ is contained in the regular subcategory ${\cal C}_{\cal F}$ of $\cal C$ generated by $\cal F$ (the case of arrows does not need a separated treatment since in any cartesian category an arrow can be identified with its graph, which is a subobject of a finite product). But ${\cal C}_{\cal F}$ is a regular category containing all the objects of $\cal F$, therefore it must be equal to ${\cal C}_{\cal F}$. This proves part (i) of the theorem. Part (ii) follows by similar arguments, using Lemma \ref{lemma:conservativity} to ensure the faithfulness of $F_{\cal F}:{{\cal C}^{\rm reg}_{\Th(T_{{\cal F}})}}^{\rm eff} \to {\cal C}$.  
\end{proof}  

\begin{remark}
If $\cal C$ is an $R$-linear abelian category then, by taking ${\cal F}^{\rm ab}$ to be the family consisting of the objects and arrows of $\cal F$ plus the arrows $r_{a}:c\to c$ for each object $c$ of $\cal F$ and each element $a\in R$ and the addition arrows $c\times c\to c$ in $\cal C$ for all the objects $c$ of $\cal F$, the category ${\cal C}_{{\cal F}^{\rm ab}}$ gets identified with the $R$-linear abelian subcategory of $\cal C$ generated by the family $\cal F$. In particular, under the hypotheses of Theorem \ref{thm:main}, taking $\cal F$ equal to the family ${\cal F}_{S}$ of objects and arrows in the image of the representation $S$ and $\cal C$ equal to $\cal A$, the theorem allows to identify, up to equivalence, the category ${\cal C}_{{\mathbb T}_{T}}$ with the $R$-linear abelian subcategory ${\cal A}_{D}$ of ${\cal A}$ generated by ${\cal F}_{S}$, provided that the functor $F_{S}$ is full on objects. Notice that ${\cal A}_{D}$ is not in general closed under isomorphisms in ${\cal A}$ (recall that we have used distinguished choices of finite limits and cokernels in $\cal A$ to define it).      
\end{remark}

Theorem \ref{thm:regsubcategory} allows, under its hypotheses, to obtain an explicit description of the regular (resp. effective regular) subcategory of a regular (resp. effective regular) category $\cal C$ generated by the family $\cal F$ not involving an inductive process, as shown by the following corollary. We shall apply this result in the abelian setting in section \ref{sec:comparisoncomodules}.

\begin{corollary}\label{cor_description_syntactic}
Let $\cal F$ be a family of objects and arrows in a category $\cal C$. Under the hypotheses of Theorem \ref{thm:regsubcategory}:
\begin{enumerate}[(i)]
\item If $\cal C$ is regular then the regular subcategory ${\cal C}_{\cal F}$ of $\cal C$ generated by $\cal F$ has as objects the images under projections $\pi_{\vec{A}}:A_{1}\times \cdots \times A_{n}\times B_{1}\times \cdots \times B_{m}\to A_{1}\times \cdots \times A_{n}$, where all the objects in $\vec{A}=(A_{1},\ldots, A_{n})$ and in $\vec{B}=(B_{1},\ldots, B_{m})$ are in $\cal F$, of equalizers of arrows $A_{1}\times \cdots \times A_{n}\times B_{1}\times \cdots \times B_{m}\to C_{1}\times \cdots \times C_{k}$ of the form $<s_{1}, \ldots, s_{k}>$, where for each $i$, $s_{i}$ is a term $A_{1}\times \cdots \times A_{n}\times B_{1}\times \cdots \times B_{m}\to C_{i}$ (where all the $C_{i}$ lie in $\cal F$). Notice that every object of ${\cal C}_{\cal F}$ is naturally equipped with a context, that is with the structure of a subobject of a finite product $A_{1}\times \cdots \times A_{n}$ of objects in $\cal F$. Conversely, every subobject of $A_{1}\times \cdots \times A_{n}$ in ${\cal C}_{\cal F}$ is, up to isomorphism, of this form. The arrows between any two such objects-in-context $S\mono A_{1}\times \cdots \times A_{n}$ and $S'\mono A_{1}'\times \cdots \times A_{n'}'$ are the arrows $S\to S'$ in $\cal C$ such that their graph is a subobject of $A_{1}\times \cdots \times A_{n}\times A_{1}'\times \cdots \times A_{n'}'$ of the above form.   

\item If $\cal C$ is effective regular then the effective regular subcategory ${\cal C}_{\cal F}$ of $\cal C$ generated by $\cal F$ is the exact completion of the category ${\cal C}_{\cal F}^{\textrm{reg}}$ described at point (i): its objects are the quotients in $\cal C$ by equivalence relations in ${\cal C}_{\cal F}^{\textrm{reg}}$ and its arrows $X\slash E\to Y\slash F$ between any two such quotients are the arrows $\alpha$ in $\cal C$ between them such that the subobject of $X\times Y$ given by the pullback of $\alpha \circ q_{E}$ along $q_{F}$, where $q_{E}:X\to X\slash E$ and $q_{F}:Y \to Y\slash F$ are the canonical projections, lies in ${\cal C}_{\cal F}^{\textrm{reg}}$.   

\item If $\cal C$ is abelian then the abelian subcategory ${\cal C}_{\cal F}$ of $\cal C$ generated by $\cal F$ has as objects the quotients of objects of the category ${\cal C}_{\cal F}^{\textrm{reg}}$ described at point (i) by subobjects in $\cal C$ in the same context (notice that in the abelian setting the description of the objects of ${\cal C}_{\cal F}^{\textrm{reg}}$ simplifies, i.e. these objects are precisely the images under projections of kernels of arrows of the form $<s_{1}, \ldots, s_{k}>:A_{1}\times \cdots \times A_{n}\times B_{1}\times \cdots \times B_{m}\to C_{1}\times \cdots \times C_{k}$), and as arrows $X\slash X'\to Y\slash Y'$ between any two such quotients the arrows $\alpha$ in $\cal C$ between them such that the subobject of $X\times Y$ given by the pullback of $\alpha \circ q_{X'}$ along $q_{Y'}$, where $q_{X'}:X\to X\slash X'$ and $q_{Y'}:Y \to Y\slash Y'$ are the canonical projections, lies in ${\cal C}_{\cal F}^{\textrm{reg}}$. 

An alternative description of ${\cal C}_{\cal F}$ is as follows. The objects of ${\cal C}_{\cal F}$ are quotients $K\slash K'$ of kernels $K$ of arrows of the form $\vec{s}=<s_{1}, \ldots, s_{k}>:A_{1}\times \cdots \times A_{n} \to C_{1}\times \cdots \times C_{k}$ by subobjects $K'$ of them which are given by images by projections $A_{1}\times \cdots \times A_{n}\times B_{1}\times \cdots \times B_{m}\to A_{1}\times \cdots \times A_{n}$ of kernels of arrows of the form $\vec{t}=<t_{1}, \ldots, t_{r}>:A_{1}\times \cdots \times A_{n}\times B_{1}\times \cdots \times B_{m} \to D_{1}\times \cdots \times D_{r}$ (where all the objects in $\vec{A}=(A_{1}, \ldots, A_{n})$, $\vec{B}=(B_{1},\ldots, B_{m})$ and $\vec{C}=(C_{1},\ldots, C_{k})$ are in $\cal F$ and the $s_{i}$ and $t_{j}$ are terms over $\Sigma_{\cal F}$). 
\begin{center} 
\begin{tikzpicture}[scale=0.95, every node/.style={scale=0.95}]
\node (0) at (0,0) {$\textrm{Ker}(\vec{t})$};
\node (1) at (5,0) {$A_{1}\times \cdots \times A_{n}\times B_{1}\times \cdots \times B_{m}$};
\node (2) at (10,0) {$D_{1}\times \cdots \times D_{r}$};
\node (3) at (1,1) {$K'$};
\node (4) at (1,2) {$K=\textrm{Ker}(\vec{s})$};
\node (5) at (5,2) {$A_{1}\times \cdots \times A_{n}$};
\node (6) at (10,2) {$C_{1}\times \cdots \times C_{k}$};
\draw[->](1) to node [above]{$\vec{t}$} (2);
\draw[right hook->](0) to node {} (1);
\draw[right hook->](4) to node {} (5);
\draw[>->](3) to node [above]{} (4);
\draw[->](1) to node [above]{$\vec{t}$} (2);
\draw[->](5) to node [above]{$\vec{s}$} (6);
\draw[->](1) to node [right]{$\pi_{\vec{A}}$} (5);
\draw[>->](3) to node [above]{ } (5);
\draw[->>](0) to node [above]{ } (3);
\end{tikzpicture}
\end{center}

The arrows $K\slash K' \to H\slash H'$ between any two such quotients, where $K\mono A_{1}\times \cdots \times A_{n}$ and $H\mono A_{1}'\times \cdots \times A_{n'}'$  are the arrows $\alpha$ in $\cal C$ between them such that the subobject of $K\times H$ given by the pullback of $\alpha \circ q_{K}$ along $q_{H}$, where $q_{K}:K\to K\slash K'$ and $q_{H}:H \to H\slash H'$ are the canonical projections, lies in ${\cal C}_{\cal F}^{\textrm{reg}}$, that is it is the image under a projection to $A_{1}\times \cdots \times A_{n}\times A_{1}'\times \cdots \times A_{n'}'$ of the kernel of an arrow of the form $<r_{1}, \ldots, r_{z}>$.    
\end{enumerate}
\end{corollary}

\begin{proof}
By Lemma D1.3.8(i) \cite{El}, any regular formula-in-context over a given signature $\Sigma$ is provably equivalent (in the empty theory) to a formula of the form $(\exists \vec{x})\phi(\vec{x}, \vec{y})$, where $\phi$ is a formula over $\Sigma$ in the same context obtained from atomic formulae by only using finite conjunctions (including the empty disjunction, yielding the truth formula $\top$). Now, if two formulae-in-context $\chi_{1}(\vec{x})$ and $\chi_{2}(\vec{x})$ are provably equivalent in the empty theory over $\Sigma$ then for any regular theory over $\Sigma$ they yield isomorphic objects $\{\vec{x}. \chi_{1}\}$ and $\{\vec{x}. \chi_{2}\}$ in its syntactic category. On the other hand, by the very definition of arrows in ${\cal C}_{\mathbb T}^{\textrm{reg}}$, one can suppose without loss of generality the formulae $\theta$ defining them to be of this form. Therefore the full subcategory of ${\cal C}_{\mathbb T}^{\textrm{reg}}$ on the formulae-in-context of this form is equivalent to ${\cal C}_{\mathbb T}^{\textrm{reg}}$. 
 
Let us now proceed to derive each of the points of the theorem from Theorem \ref{thm:regsubcategory}. 

(i) The above remark implies that the objects in the image of the functor $F_{\cal F}$ are precisely the interpretations in the structure $T_{\cal F}$ of the formulae of the form $(\exists \vec{x})\phi(\vec{x}, \vec{y})$, where $\phi$ is a formula over $\Sigma$ in the same context obtained from atomic formulae by only using finite conjunctions (including the empty disjunction, yielding the truth formula $\top$). Now, an atomic formula $\theta(\vec{x}, \vec{y})$ over $\Sigma_{\cal F}$, where $\vec{x}=(x_{1}^{A_{1}}, \ldots, x_{n}^{A_{n}})$ and $\vec{y}=(y_{1}^{B_{1}}, \ldots, y_{m}^{B_{m}})$, is a formula of the form $s=t$, where $s$ and $t$ are terms $A_{1}\times \cdots A_{n}\times B_{1}\times \cdots \times B_{m}\to C$ over $\Sigma_{\cal F}$. Its interpretation is the subobject of $A_{1}\times \cdots A_{n}\times B_{1}\times \cdots \times B_{m}$ given by the the equalizer of $s$ and $t$, while the interpretation of a finite conjunction $s_{1}=t_{1}\wedge \cdots \wedge s_{k}=t_{k}$ is the intersection of the equalizers $Eq(s_{i}, t_{i})$ (for all $i=1, \ldots, k$), equivalently the equalizer of the arrows $<s_{1}, \ldots, s_{k}>, <t_{1}, \ldots, t_{k}>:A_{1}\times \cdots A_{n}\times B_{1}\times \cdots \times B_{m} \to C_{1}\times \cdots \times C_{k}$. On the other hand, the interpretation of the formula $(\exists \vec{x})(s_{1}=t_{1} \wedge \cdots \wedge s_{k}=t_{k})(\vec{x}, \vec{y})$ is given by the image of composite of the subobject $Eq(<s_{1}, \ldots, s_{k}>, <t_{1}, \ldots, t_{k}>)\mono A_{1}\times \cdots A_{n}\times B_{1}\times \cdots \times B_{m}$ with the canonical projection $\pi_{\vec{A}}:A_{1}\times \cdots \times A_{n}\times B_{1}\times \cdots \times B_{m}\to A_{1}\times \cdots \times A_{n}$. From this, the description of the objects of ${\cal C}_{\cal F}$ follows immediately. The description of arrows follows immediately from that of objects, using the identification between an arrow and its graph, and the fact that the model $T_{\cal F}$ is conservative for the theory $\Th(T_{\cal F})$.

The fact that every object of ${\cal C}_{\cal F}$ is naturally equipped with a context, that is with the structure of a subobject of a finite product $A_{1}\times \cdots \times A_{n}$ of objects in $\cal F$ and that, conversely, every subobject of $A_{1}\times \cdots \times A_{n}$ in ${\cal C}_{\cal F}$ is, up to isomorphism, of this form follows from the fact that any object of the syntactic category of a regular theory is a regular-formula-in-context $\{\vec{x}. \phi\}$ over its signature, which defines a subobject of $\{\vec{x}.\top\}$, and every subobject of $\{\vec{x}.\top\}$ in this category is isomorphic to one of this form.    

(ii) This follows at once from point $(i)$ by using the description of the exact completion of a regular category given in section \ref{sec:exactcompletion}.

(iii) This follows from point (ii) by observing that:
\begin{itemize}
\item Since terms can be subtracted in an additive context,
the description of the objects of ${\cal C}_{\cal F}$ simplifies, i.e. the objects of ${\cal C}_{\cal F}$ are precisely the images under projections of kernels of arrows of the form $<s_{1}, \ldots, s_{k}>:A_{1}\times \cdots \times A_{n}\times B_{1}\times \cdots \times B_{m}\to C_{1}\times \cdots \times C_{k}$; 

\item Quotients by equivalence relations are, in the abelian setting, quotients by subobjects, and in the regular syntactic category of a regular theory $\mathbb T$, the subobjects of an object $\{\vec{x}. \phi\}$ are, up to isomorphism, all of the form $\{\vec{x}. \psi\}\mono \{\vec{x}. \phi\}$, where $\psi$ is a formula in the context $\vec{x}$ which $\mathbb T$-provably implies $\phi$, in other words, $T_{\cal F}$ being conservative, they correspond exactly to the definable subobjects of $[[\vec{x}. \phi]]_{T_{\cal F}}$ in $\cal C$.
\end{itemize}

The alternative description of ${\cal C}_{\cal F}$ given in point (iii) follows from the fact that every object of ${\cal C}_{\cal F}^{\textrm{reg}}$ is a quotient (in ${\cal C}_{\cal F}^{\textrm{reg}}$) of an object which is the kernel of an arrow of the form $<s_{1}, \ldots, s_{k}>$, from the characterization of subobjects of finite products of objects in $\cal F$ provided by point (i) of the theorem and from the fact that ${\cal C}_{\cal F}^{\textrm{reg}}$ is closed under subobjects in its effectivization (see section \ref{sec:preliminaries}).   
\end{proof}

\subsection{The (in)dependence from $T$}

Theorem \ref{thm:main} shows that, given two representations $T:D \to \Rmod$ and $T':D'\to \Rpmod$, the categories ${\cal C}_{{\mathbb T}_{T}}$ and ${\cal C}_{{\mathbb T}_{T'}}$ are equivalent if and only if the theories $\Th(T)$ and $\Th(T')$ are Morita-equivalent, i.e. they have equivalent classifying toposes (recall from section \ref{sec:preliminaries} the intrinsic characterization of the effectivization of the regular syntactic category of a regular theory as the full subcategory of its classifying topos on the supercoherent objects). 

Particularly interesting is the case when $D=D'$. In this case, by the conservativity of universal models in syntactic categories, we have that there is an equivalence ${\cal C}_{{\mathbb T}_{T}}\simeq {\cal C}_{{\mathbb T}_{T'}}$ compatible with the canonical representations of $D$ in the two syntactic categories if and only if $\Th(T)=\Th(T')$ as theories over the signature $L_{D}$, in other words if and only if for every regular sequent $\sigma$ over $L_{D}$, $\sigma$ is satisfied by $T$ if and only if it is satisfied by $T'$. Notice that if $D$ is Nori's diagram of motives, the exactness conditions are expressible as regular sequents over $L_{D}$, whilst the bounds on dimension are not. This contrasts with the usual feeling that in order to prove the independence of a category of motives from any of its realizations the knowledge of the dimensions of the cohomology groups be essential.

We can reformulate the condition $\Th(T)=\Th(T')$ algebraically by using the explicit description of the interpretations of regular formulae over $L_{D}$ in a regular category obtained in section \ref{sec:subcategories}, as follows. Since $T$ (resp. $T'$) is a conservative model of $\Th(T)$ (resp. of $\Th(T')$), a regular sequent $(\psi \vdash_{\vec{x}} \phi)$ over $L_{D}$ is provable in $\Th(T)$ (resp. in $\Th(T')$) if and only if it is valid in $T$ (resp. in $T'$), that is if and only if $[[\vec{x}.\psi]]_{T}\subseteq [[\vec{x}.\psi]]_{T}$ (resp. $[[\vec{x}.\psi]]_{T'}\subseteq [[\vec{x}.\psi]]_{T'}$). Now, the objects of the form $[[\vec{x}.\psi]]_{T}$ are precisely the images under canonical projections $p^{T}_{\vec{d}, \vec{e}}:T(d_{1})\times \cdots \times T(d_{n})\times T(e_{1})\times \cdots \times T(e_{m})\to T(d_{1})\times \cdots \times T(d_{n})$ of kernels of arrows $T(d_{1})\times \cdots \times T(d_{n})\times T(e_{1})\times \cdots \times T(e_{m}) \to T(c_{1})\times \cdots \times T(c_{k})$ whose components are linear combinations of arrows of the form $T(f)$ for $f$ an edge of $D$ with coefficients in $R$, and similarly for $T'$. 

Summarizing, we have the following

\begin{theorem}\label{thm:independence}
	
\begin{enumerate}[(i)]
\item Given two representations $T:D \to \Rmod$ and $T':D'\to \Rpmod$ (or more generally, any pair of representations of diagrams in effective regular categories), the categories ${\cal C}_{{\mathbb T}_{T}}$ and ${\cal C}_{{\mathbb T}_{T'}}$ are equivalent if and only if the theories $\Th(T)$ and $\Th(T')$ are Morita-equivalent.	
		
\item Let $R$ be a ring, $D$ a diagram and $T, T':D \to \Rmod$ representations of $D$. Then we have an equivalence of categories $\xi:{\cal C}_{{\mathbb T}_{T}} \to {\cal C}_{{\mathbb T}_{T'}}$ making the diagram
\begin{center}
\begin{tikzpicture}
\node (1) at (3,0) {${\cal C}_{{\mathbb T}_{T}}$};
\node (6) at (3,-2) {${\cal C}_{{\mathbb T}_{T'}}$};
\node (2) at (0, -2) {$D$};
\draw[->](2) to node [above]{$\tilde{T}$} (1);
\draw[->](2) to node [above]{$\tilde{T'}$} (6);
\draw[dashed, ->](1) to node [right]{$\xi$} (6);
\end{tikzpicture}
\end{center}
commute if and only if $\Th(T)=\Th(T')$, equivalently if and only if for any tuples $$<s_{1}, \ldots, s_{k}>:d_{1}, \ldots, d_{n}, e_{1}, \ldots, e_{m}\to c_{1}, \ldots, c_{k}$$ and $$<s_{1}', \ldots, s_{k'}'>:d_{1}, \ldots, d_{n}, e_{1}, \ldots, e_{m}\to c_{1}', \ldots, c_{k'}'$$ of $R$-linear combinations of edges in $D$,
\[
p^{T}_{\vec{d}, \vec{e}}(\textrm{Ker}(T(<s_{1}, \ldots, s_{k}>))\subseteq p^{T}_{\vec{d}, \vec{e}}(\textrm{Ker}(T(<s_{1}', \ldots, s_{k'}'>))
\]
as subobjects of $T(d_{1})\times \cdots \times T(d_{n})$ in $\Rmod$ if and only if 
\[
p^{T'}_{\vec{d}, \vec{e}}(\textrm{Ker}(T'(<s_{1}, \ldots, s_{k}>))\subseteq p^{T'}_{\vec{d}, \vec{e}}(\textrm{Ker}(T'(<s_{1}', \ldots, s_{k'}'>))
\]
as subobjects of $T'(d_{1})\times \cdots \times T'(d_{n})$ in $\Rmod$.  
\end{enumerate}
\end{theorem}\qed

This result is the beginning of a possible answer to the question of the independence of $\ell$ in $\ell$-adic cohomology theory as phrased for instance in section 11 of \cite{LL}. A fundamentally new input appears in the present paper: the $\mathbb Q$-linear abelian category associated to $\ell$-adic cohomology is no more naively defined as the subcategory of ``geometric'' objects and morphisms in the category of $\ell$-adic representations of the Galois group of the base field.  

\subsection{The relationship with the category $\textbf{Comod}_{\textrm{fin}}(\End^{\vee}(T))$}\label{sec:comparisoncomodules}

Given a field $k$, let $k\textrm{-vect}_{\textrm{f}}$ denote the category of finite-dimensional vector spaces over $k$.
 
In this section we shall compare our syntactic construction of Nori's category with Nori's original construction. We already know, from the universal property of Theorem \ref{thm:main}, that they are canonically equivalent, but it is instructive to see this directly. In fact, this will lead to a concrete description of the objects and arrows of Nori's category in terms of the objects and edges of the diagram, improving that given in \cite{Arapura} (cf. Lemma 2.2.5 therein).

The canonical representation $i_{D}:D\to \textrm{Comod}_{\textrm{fin}}(\End^{\vee}(T))$ is given by the assignment $d \leadsto  T(d)$, where $T(d)$ is endowed with the canonical structure of comodule over $\End^{\vee}(T)$. 

\begin{theorem}\label{thm:equivcat}
Let $D$ be a diagram and $T:D\to k\textrm{-vect}_{\textrm{f}}$ a representation. Then there exists an equivalence of categories $\chi:{\cal C}_{{\mathbb T}_{T}}\to \textrm{Comod}_{\textrm{fin}}(\End^{\vee}(T))$ compatible with the canonical representations $\tilde{T}$ of $D$ in ${\cal C}_{{\mathbb T}_{T}}$ and $i_{D}$ of $D$ in $\textrm{Comod}_{\textrm{fin}}(\End^{\vee}(T))$:
\begin{center}
\begin{tikzpicture}
\node (1) at (3,0) {${\cal C}_{{\mathbb T}_{T}}$};
\node (6) at (3,-3) {$\textrm{Comod}_{\textrm{fin}}(\End^{\vee}(T))$};
\node (2) at (0, -3) {D};
\draw[->](2) to node [above]{$\tilde{T}$} (1);
\draw[->](2) to node [above]{$i_{D}$} (6);
\draw[dashed, ->](1) to node [right]{$\chi$} (6);
\end{tikzpicture}
\end{center}
 
\end{theorem}

\begin{proof}
The fact that $T$ is equal to the the composite of the forgetful functor $\textrm{Comod}_{\textrm{fin}}(\End^{\vee}(T) \to k\textrm{-vect}_{\textrm{f}}$ with the representation $i_{D}$ ensures, by Theorem \ref{thm:main}, the existence of a unique faithful exact functor $\chi:{\cal C}_{{\mathbb T}_{T}}\to \textrm{Comod}_{\textrm{fin}}(\End^{\vee}(T))$ making the above triangle commute. To prove that $\chi$ is an equivalence, it remains to show that it is full and essentially surjective. We can clearly check these conditions on each of the full subcategories $\End(T|_{F})\textrm{-mod}_{\textrm{fin}}$ which cover $
\textrm{Comod}_{\textrm{fin}}(\End^{\vee}(T))=\colim{F\subseteq D \textrm{ finite }}\End(T|_{F})\textrm{-mod}_{\textrm{fin}}$.

Any finite-dimensional (over $k$) $\End(T|_{F})$-module $M$ is clearly a quotient of $\End(T|_{F})^{n}$ for some $n$ in $\End(T|_{F})\textrm{-mod}_{\textrm{fin}}$, and any $\End(T|_{F})$-equivariant map $f:M\to N$ between $\End(T|_{F})$-modules $M$ and $N$ which are quotients of finite powers of $\End(T|_{F})$ via maps $p:\End(T|_{F})^{n}\epi M$ and $q:\End(T|_{F})^{m}\epi M$ lifts to a $\End(T|_{F})$-equivariant map $\tilde{f}:\End(T|_{F})^{n} \to \End(T|_{F})^{m}$ such that $q\circ \tilde{f}=f\circ p$. Moreover, given a quotient map $p:\End(T|_{F})^{n}\epi M$, since the kernel of $p$ is of finite dimension, it is itself the image of a finite power of $\End(T|_{F})$, by a surjective homomorphism $r:\End(T|_{F})^{k}\to ker(p)$. The kernel of $p$ can thus be identified with the image factorisation in $\End(T|_{F})\textrm{-mod}_{\textrm{fin}}$ of a morphism $\End(T|_{F})^{k} \to \End(T|_{F})^{n}$, and $M$ is isomorphic to the quotient of $\End(T|_{F})^{n}$ by it.   

Therefore, to prove the essential surjectivity of the functor $\chi$ it suffices to prove that each algebra $\End(T|_{F})$ is, when considered as an object of the category $\End(T|_{F})\textrm{-mod}_{\textrm{fin}}$, isomorphic to an object in the image of $\chi$ and that every  endomorphism of it is, under this identification, the image under $\chi$ of an endomorphism in ${\cal C}_{{\mathbb T}_{T}}$.  

To identify $\End(T|_{F})$ with an object in the image of $\chi$ we choose, for each object $d$ of $D$, a basis of the finite-dimensional $k$-vector space $T(d)$. This induces isomorphisms of left $\End(T|_{F})$-modules
\[
\mathbin{\mathop{\textrm{ $\prod$}}\limits_{d\in D}} \Hom_{k}(T(d), T(d))\cong \mathbin{\mathop{\textrm{ $\prod$}}\limits_{d\in D}} T(d)^{\textrm{dim}(T(d))} 
\] 
and
\[
\mathbin{\mathop{\textrm{ $\prod$}}\limits_{f:d'\to d'' \textrm{ in } D}} \Hom_{k}(T(d'), T(d''))\cong \mathbin{\mathop{\textrm{ $\prod$}}\limits_{f:d'\to d'' \textrm{ in } D}} T(d'')^{\textrm{dim}(T(d'))}. 
\]

Notice that the objects on the right-hand-side of these isomorphisms are both in the image of the functor $\chi$. Let us now show that, under these isomorphisms, the arrow
\[
S:\mathbin{\mathop{\textrm{ $\prod$}}\limits_{d\in D}} \Hom_{k}(T(d), T(d)) \to \mathbin{\mathop{\textrm{ $\prod$}}\limits_{f:d'\to d'' \textrm{ in } D}} \Hom_{k}(T(d'), T(d''))
\]   
considered in section \ref{sec:review} corresponds to a definable morphism between these objects. This is true since for every edge $f:d'\to d''$ in the diagram $D$ both the arrows
\[
T(f) \circ - :T(d')^{\textrm{dim}(T(d'))} \to T(d'')^{\textrm{dim}(T(d'))}
\]
and 
\[
\tilde{f}:T(d'')^{\textrm{dim}(T(d''))} \to T(d'')^{\textrm{dim}(T(d'))},
\]
where $\tilde{f}$ is the homomorphism induced by the $\textrm{dim}(T(d')) \times \textrm{dim}(T(d''))$-matrix $M_{f}$   given by the representation of the homomorphism $T(f):T(d')\to T(d'')$ in the chosen bases for $T(d')$ and $T(d'')$ (more explicitly, $\tilde{f}$ assigns to a tuple $(a_{1}, \ldots, a_{\textrm{dim}(T(d''))})\in T(d'')^{\textrm{dim}(T(d''))}$ the tuple 
\[
\big( \mathbin{\mathop{\textrm{ $\sum$}}\limits_{j=1, \ldots, \textrm{dim}(T(d''))} k^{f}_{(1,j)}}a_{j}, \ldots,  \mathbin{\mathop{\textrm{ $\sum$}}\limits_{j=1, \ldots, \textrm{dim}(T(d''))} k^{f}_{(\textrm{dim}(T(d')),j)}}a_{j} \big ) \in T(d'')^{\textrm{dim}(T(d'))},
\]
where the $k^{f}_{(i,j)}$ are the coefficients in $k$ of the matrix $M_{f}$) are definable over $L_{D}$. Indeed, the language $L_{D}$ contains a function symbol for each edge $f$ in $D$, a unary function symbol for each element of the field $k$ and the binary function symbols $+$ on each sort $d$ formalizing the sum operation on $T(d)$.

Now that we have identified the algebra $\End(T|_{F})$ with an object in the image of $\chi$, to conclude the proof of the essential surjectivity of $\chi$ it remains to show that every endomorphism of it as a \emph{left} $\End(T|_{F})$-module is definable, i.e. it comes from an endomorphism of ${\cal C}_{{\mathbb T}_{T}}$. Such endomorphisms correspond exactly to the elements of the algebra $\End(T|_{F})$, via the assignment sending such an element to the endomorphism of multiplication on the \emph{right} by it. 
 
Now, an element of $\End(T|_{F})$ is a collection of endomorphisms $\alpha_{d}$ on the $T(d)$ (for each $d\in D$). The associated endomorphism is definable since it acts at each component $d$ as the composition with $\alpha_{d}$ on the right, which is given by the multiplication with the matrix obtained by representing $\alpha_{d}$ in terms of the chosen basis for $T(d)$. This completes the proof of the essential surjectivity of $\chi$.

Let us now prove that the functor $\chi$ is full. The objects in the image of $\chi$ are quotients $K\slash K'$, where $K$ and $K'$ are in the following diagram (cf. Corollary \ref{cor_description_syntactic}):  
\begin{center}   
\begin{tikzpicture}[scale=0.9, every node/.style={scale=0.9}]
\node (0) at (0,0) {$\textrm{Ker}(T(\vec{t}))$};
\node (1) at (5,0) {$T(a_{1})\times \cdots \times T(a_{n})\times T(b_{1})\times \cdots \times T(b_{m})$};
\node (2) at (11,0) {$T(d_{1})\times \cdots \times T(d_{r})$};
\node (3) at (1,1) {$K'$};
\node (4) at (1,2) {$K=\textrm{Ker}(T(\vec{s}))$};
\node (5) at (5,2) {$T(a_{1})\times \cdots \times T(a_{n})$};
\node (6) at (11,2) {$T(c_{1})\times \cdots \times T(c_{k})$};
\draw[->](1) to node [above]{$T(\vec{t})$} (2);
\draw[right hook->](0) to node {} (1);
\draw[right hook->](4) to node {} (5);
\draw[>->](3) to node [above]{} (4);
\draw[->](1) to node [above]{$T(\vec{t})$} (2);
\draw[->](5) to node [above]{$T(\vec{s})$} (6);
\draw[->](1) to node [right]{$\pi$} (5);
\draw[>->](3) to node [above]{ } (5);
\draw[->>](0) to node [above]{ } (3);
\end{tikzpicture}
\end{center}
where all the objects $a_{i}, b_{j}, c_{m}, d_{l}$ are in $D$ and all the terms in $\vec{s}$ and in $\vec{t}$ are over $L_{D}$.  

We observe that if $F$ is the full subdiagram of $D$ on any set of objects of $D$ containing all the objects of the form $a_{i}, b_{j}, c_{m}$ or $d_{l}$ then $K\slash K'$ is an $\End(T|_{F})$-module. Indeed, this follows at once from the naturality of endomorphisms of $T$, as model endomorphisms of the $L_{D}$-structure $T$ (cf. section \ref{sec:regularlogic}).  

Since the colimit $
\textrm{Comod}_{\textrm{fin}}(\End^{\vee}(T))=\colim{F\subseteq D \textrm{ finite }}\End(T|_{F})\textrm{-mod}_{\textrm{fin}}$ is filtered and each of the subcategories $\End(T|_{F})\textrm{-mod}_{\textrm{fin}}$ is full in the category $\textrm{Comod}_{\textrm{fin}}(\End^{\vee}(T))$, we can suppose without loss of generality that our arrow $\alpha:K\slash K' \to H\slash H'$ whose domain and codomain are in the image of $\chi$ lies in the category $\End(T|_{F})\textrm{-mod}_{\textrm{fin}}$ for some $F$ containing all the objects appearing in the diagrams defining $K$, $K'$, $H$ and $H'$, i.e. that it is a $k$-linear $\End(T|_{F})$-equivariant map. To prove that $\alpha$ is in the image of $\chi$ we shall argue as follows. As we did for proving the essential surjectivity of $\chi$, we choose a basis $\{   b_{1}^{d}, \ldots, b^{d}_{r_{d}}\}$ for each $T(d)$ (as $k$-vector space) for $d\in F$ as well as bases $\{v_{1}, \ldots, v_{p}\}$ for $K\slash K'$ and $\{w_{1}, \ldots, w_{q}\}$ for $H\slash H'$ over $k$. Since $K\slash K'$ is a $\End(T|_{F})$-module, the choice of the basis $\{v_{1}, \ldots, v_{p}\}$ for it induces a surjective $\End(T|_{F})$-equivariant homomorphism $u:\End(T|_{F})^{p}\to K\slash K'$ sending the canonical basis of $\End(T|_{F})^{p}$ to the basis $\{v_{1}, \ldots, v_{p}\}$; similarly, we have a $\End(T|_{F})$-equivariant homomorphism $v:\End(T|_{F})^{q}\to H\slash H'$  sending the canonical basis of $\End(T|_{F})^{q}$ to the basis $\{w_{1}, \ldots, w_{q}\}$ of $H\slash H'$. Since our arrow $\alpha$ is $\End(T|_{F})$-equivariant, it lifts to an $\End(T|_{F})$-equivariant homomorphism $\tilde{\alpha}:\End(T|_{F})^{p} \to \End(T|_{F})^{q}$ making the following diagram commute: 
\[  
\xymatrix {
\End(T|_{F})^{p} \ar[d]^{\tilde{\alpha}} \ar[r]^{u} &  K\slash K' \ar[d]^{\alpha} \\
\End(T|_{F})^{q}  \ar[r]^{v} & H\slash H'. } 
\]

By the first part of the proof, the kernel of $u$ and that of $v$ can both be identified with objects in the image of $\chi$. To prove that $\alpha$ is in the image of $\chi$ as well, it suffices to prove that $u$ and $v$ are, using the identification of $\End(T|_{F})^{p}$ with a definable subobject of $(\mathbin{\mathop{\textrm{ $\prod$}}\limits_{d\in F}} T(d)^{r_{d}})^{p}$ and of $\End(T|_{F})^{q}$ with a definable subobject of $(\mathbin{\mathop{\textrm{ $\prod$}}\limits_{d\in F}} T(d)^{r_{d}})^{q}$ (induced by the above choice of basis for the $T(d)$). Let us prove this for $u$, the argument for $v$ being perfectly analogous. To this end, we notice that under the identification of $\End(T|_{F})$ with a definable subobject of $\mathbin{\mathop{\textrm{ $\prod$}}\limits_{d\in F}} T(d)^{r_{d}}$, an element $\xi\in \End(T|_{F})$ corresponds to the function $f_{\xi}$ such that $f_{\xi}(d)(i)=\xi(d)(b^{d}_{i})$ for all $d\in F$ and $i\in \{1, \ldots, r_{d}\}$. Now, since $K$ is a subobject of $T(a_{1})\times \cdots \times T(a_{n})$, we can express each of the basis vectors $v_{j}$ as a $k$-linear combination of the basis elements $b^{d}_{i}$ (for $d\in F$ and $i\in \{1, \ldots, r_{d}\}$), say $v_{j}=\mathbin{\mathop{\textrm{ $\sum$}}\limits_{d\in F, i=1, \ldots, r_{d}}} k^{j}_{(d, i)}b^{d}_{i}$. Now, the arrow $u$ sends a tuple $(\xi_{1}, \ldots, \xi_{p})\in \End(T|_{F})^{p}$ to the element $\xi_{1}(v_{1})+ \cdots + \xi_{p}(v_{p})$. But we have 

{\flushleft $\xi_{1}(v_{1})+ \cdots + \xi_{p}(v_{p})=\mathbin{\mathop{\textrm{ $\sum$}}\limits_{d\in F, i=1, \ldots, r_{d}}} k^{1}_{(d, i)}\xi_{1}(d)(b^{d}_{i})+ \cdots + \mathbin{\mathop{\textrm{ $\sum$}}\limits_{d\in F, i=1, \ldots, r_{d}}} k^{p}_{(d, i)}\xi_{p}(d)(b^{d}_{i})
$ $= \mathbin{\mathop{\textrm{ $\sum$}}\limits_{d\in F, i=1, \ldots, r_{d}}} k^{1}_{(d, i)}f_{\xi_{1}}(d)(i)+ \cdots + \mathbin{\mathop{\textrm{ $\sum$}}\limits_{d\in F, i=1, \ldots, r_{d}}} k^{p}_{(d, i)}f_{\xi_{p}}(d)(i)$,} which shows that $u$ is definable.   
\end{proof}

The following result is an immediate consequence of Theorem \ref{thm:equivcat} and Corollary \ref{cor_description_syntactic}(iii).

\begin{corollary}\label{cor:explicitdescriptionobjects}
Let $D$ be a diagram and $T:D\to k\textrm{-vect}_{\textrm{f}}$ a representation. Let $i_{D}:D\to\textrm{Comod}_{\textrm{fin}}(\End^{\vee}(T))$ be the canonical representation. Then the objects of $\textrm{Comod}_{\textrm{fin}}(\End^{\vee}(T))$ are, up to isomorphism, quotients of the form $K\slash K'$, where $K$ and $K'$ sit in a diagram:
\begin{center}   
\begin{tikzpicture}[scale=0.88, every node/.style={scale=0.88}]
\node (0) at (0,0) {$\textrm{Ker}(i_{D}(\vec{t}))$};
\node (1) at (5,0) {$i_{D}(a_{1})\times \cdots \times i_{D}(a_{n})\times i_{D}(b_{1})\times \cdots \times i_{D}(b_{m})$};
\node (2) at (11,0) {$i_{D}(d_{1})\times \cdots \times i_{D}(d_{r})$};
\node (3) at (1,1) {$K'$};
\node (4) at (1,2) {$K=\textrm{Ker}(i_{D}(\vec{s}))$};
\node (5) at (5,2) {$i_{D}(a_{1})\times \cdots \times i_{D}(a_{n})$};
\node (6) at (11,2) {$i_{D}(c_{1})\times \cdots \times i_{D}(c_{k})$};
\draw[->](1) to node [above]{$i_{D}(\vec{t})$} (2);
\draw[right hook->](0) to node {} (1);
\draw[right hook->](4) to node {} (5);
\draw[>->](3) to node [above]{} (4);
\draw[->](1) to node [above]{$i_{D}(\vec{t})$} (2);
\draw[->](5) to node [above]{$i_{D}(\vec{s})$} (6);
\draw[->](1) to node [right]{$\pi$} (5);
\draw[>->](3) to node [above]{ } (5);
\draw[->>](0) to node [above]{ } (3);
\end{tikzpicture}
\end{center}
where all the objects $a_{i}, b_{j}, c_{m}, d_{l}$ are in $D$ and all the terms in $\vec{s}$ and in $\vec{t}$ are over $L_{D}$.  

Moreover, any subobject of $i_{D}(a_{1})\times \cdots \times i_{D}(a_{n})$ in $\textrm{Comod}_{\textrm{fin}}(\End^{\vee}(T))$ is isomorphic to one of the form $K'\mono i_{D}(a_{1})\times \cdots \times i_{D}(a_{n})$ specified in the above diagram.
\end{corollary}\qed

\begin{remarks}
\begin{enumerate}[(a)]
\item In light of Theorem \ref{thm:main}, Theorem \ref{thm:equivcat} provides an alternative proof of the classical Nori's Tannakian theorem, and also of the classical theorem in Tannaka duality (cf. \cite{SR}, \cite{DeligneMilne} and \cite{JS}) asserting that any $k$-linear abelian category $\cal A$ with an exact faithful functor $U:{\cal A}\to k\textrm{-vect}_{\textrm{f}}$ is equivalent to the category $\textrm{Comod}_{\textrm{fin}}(\End^{\vee}(U))$ (this latter result follows by taking $D$ equal to the underlying diagram of $\cal A$ and applying the universal property of Theorem \ref{thm:main}).

\item As it is clear from its proof, Theorem \ref{thm:equivcat} generalizes to a Noetherian ring $R$, replacing $k\textrm{-vect}_{\textrm{f}}$ with $\Rmod_{f}$, bases with systems of generators, and $\textrm{Comod}_{\textrm{fin}}(\End^{\vee}(T))$ with $\colim{F \subseteq D \textrm{ finite}}\End (T|_{F})\text{-{\rm mod}}_{\rm fin}$.

\item In \cite{HMS} (Proposition B.9) and \cite{Arapura} (Lemma 2.2.5) it is observed that every object of the category $\textrm{Comod}_{\textrm{fin}}(\End^{\vee}(T)$ is a subquotient (i.e., a subobject of a quotient) of a finite direct sum of objects of the form $i_{D}(d)$ (for $d\in D$). It follows from Corollary \ref{cor:explicitdescriptionobjects} that the following more specific description holds: every object $A$ of $\textrm{Comod}_{\textrm{fin}}(\End^{\vee}(T))$ fits in an exact sequence  
\[
0 \rightarrow K' \mono K \to A \to 0,
\]
where $K$ and $K'$ are as in the statement of the corollary. 

\end{enumerate}
\end{remarks}

\section{Application to motives}\label{sec:applications}

In this section we discuss the application to motives of the general construction presented above, also in relation to the classical construction of Nori motives.

\subsection{Nori's diagram and Betti homology}

Nori applies his general construction of the category ${\cal C}_{T}$ to propose a definition of a category of (mixed) motives over a given base field. To this end, he has to choose a diagram $D$ defined in terms of schemes over that base field, a ring or coefficients field $R$ and a representation $$T:D\to \Rmod_{\rm f}$$
given by a (co)homological functor.

Since the (co)homological functor must take values in the finite-type modules over $R$, and one looks for a category of motives with coefficients in $R={\mathbb Z}$ or $R={\mathbb Q}$, Nori proposes a definition based on Betti homology.

\begin{definition}\label{defNori}
Given a subfield $K$ of $\mathbb C$, let $D$ be the diagram defined as follows:
\begin{itemize}
	\item the objects of $D$ are the triples $(X,Y,i)$ consisting in an element $X$ of a set of representatives of separated finite-type schemes over $K$, a closed subscheme $Y$ of $X$ and a non-negative integer $i\in {\mathbb N}$;
	
	\item The edges of $D$ are on one hand of the form $$(X,Y,i) \to (X',Y',i) \quad\quad(\forall i\geq 0),$$ 
	for any commutative square
	\[  
	\xymatrix@1@=3pt@M=3pt {
		& & & & & X & \to & X'  & & & & &\\
		& & & & & & &  & &  & \\
		& & & & & Y \ar@{^{(}->}[uu] & \to & Y' \ar@{^{(}->}[uu]}
	\]
    and on the other hand of the form $$(X,Y,i)\to (Y,Z,i-1) \quad\quad(\forall i\geq 1),$$ for each sequence of closed subscheme inclusions
    $$Z\hookrightarrow Y \hookrightarrow X.$$ 
\end{itemize}
Taking $R={\mathbb Z}$ or $R={\mathbb Q}$, consider the representation
$$T:D\to \Rmod_{\rm f}$$
which assigns to each triplet $(X,Y,i)$ its relative Betti homology $H_{i}(X, Y)$ with coefficients in $R$, endowed with the functorial homomorphisms and the boundary homomorphisms. The category of Nori motives is by definition the abelian $R$-linear category ${\cal C}_{T}$ associated with this choice of $D$ and $T$.
\end{definition}

The interest of this definition is illustrated for instance by the following theorem.

\begin{theorem}[\cite{ABV}]
In Definition \ref{defNori} above, take $R={\mathbb Z}$ and restrict the diagram $D$ to triplets $(X, Y, i)$ such that $i=0$ (resp. $i\leq 1$). Then:
\begin{enumerate}[(i)]
	\item In the case of the condition $i=0$, the universal Nori abelian category ${\cal C}_{T}$ is equivalent to the abelian category of Artin $0$-motives, that is to the category of constructible \'etale sheaves of abelian groups on $\textrm{Spec}(K)$.
	
	\item In the case of the condition $i\leq 1$, the universal Nori abelian category ${\cal C}_{T}$ is equivalent to the abelian category of Deligne $1$-motives with torsion. The objects of this category are complexes of abelian group schemes on $\textrm{Spec}(K)$ of the form ${\cal F}\to {\cal G}$, where
	\begin{itemize}
		\item $\cal F$ is a constructible \'etale sheaf of abelian groups on $\textrm{Spec}(K)$,
		\item $\cal G$ is a semi-abelian variety on $\textrm{Spec}(K)$.  
	\end{itemize}
	 The arrows are obtained from complex morphisms by formally inverting quasi-isomorphisms.
	\end{enumerate}	
\end{theorem}

Another illustration of the interest of Nori's definition is provided by the following theorem.

\begin{theorem}[\cite{Kontsevich} and \cite{HMS}]
	In Definition \ref{defNori} above, take as coefficient field $R={\mathbb Q}$. Let us start from the diagram consisting of the triplets $(X, Y, i)$ endowed with the product defined by $$(X, Y, i)\times (X', Y', i')=(X\times Y' \sqcup X'\times Y, i+i')$$ and take $D$ equal to the diagram obtained from it by formally inverting the product by the object $({\mathbb G}_{1}, \{1\}, 1)$. Then
	\begin{enumerate}[(i)]
		\item The universal Nori abelian category ${\cal C}_{T}$ endowed with the functor $F_{T}:{\cal C}_{T}\to {\mathbb Q}\textrm{-vect}_{\textrm{f}}$ of relative Betti homology has naturally the structure of a rigid Tannakian category endowed with a fiber functor. It defines a motivic Galois group $G_{\textrm{mot}}$. 
		
		\item If the base field is $K={\mathbb Q}$ and $P$ is the algebra of Kontsevich-Zagier ``formal periods'' (in the sense of \cite{KZ}), $\textrm{Spec}(P)$ is a torsor under the action of $G_{\textrm{mot}}$.  
	\end{enumerate}
\end{theorem}

\subsection{A general criterion for the existence of categories of mixed motives}

We shall work over a base field $K$ of arbitrary characteristic. We take for coefficient field $R={\mathbb Q}$ or possibly $R={\mathbb Z}$. We take as $D$ a diagram deduced from the category of finite-type separated schemes over $K$ (as for instance in Definition \ref{defNori}) and as a representation $T:D\to {\mathbb Q}\textrm{-vect}$ or $T:D\to {\mathbb Z}\textrm{-vect}$ a cohomological or homological functor with coefficients of characteristic $0$ (including in fields much larger than $\mathbb Q$ such that the $\ell$-adic fields). 
 
Concretely, the objects of $D$ consist of a geometric part and of an index:
\begin{itemize}
	\item The geometric part of an object of $D$ is an element of a set of representatives of finite-type separated schemes over $K$ (possibly subject to conditions such as: smooth, projective, etc.), or a pair of inclusions of such schemes $Y\hookrightarrow X$ as in Definition \ref{defNori}, or possibly more complex diagrams with values in the category of finite-type separated schemes over $K$.
	
	\item The index which completes the geometric part of an object of $D$ consists at least of an integer $i\geq 0$ which will be the degree of the homology or cohomology spaces associated with the object. It can possibly be completed by another integer $j\in {\mathbb Z}$ which will allow a Tate-type torsion of coefficients of cohomology spaces (assigning, for instance, to a scheme completed by a double index $(i, j)$ the space of $\ell$-adic cohomology $H^{i}(\overline{X}, {\mathbb Q}_{\ell}(j))$ or Bloch's higher Chow group $\textrm{CH}^{j}(X, 2j-i)_{\mathbb Q}$). It could also be completed by a binary index which would serve for assigning to the object a cohomology space when the index takes the first value and a homology space (or a cohomology with compact support space) when the index takes the other value: this would allow to consider cohomology and homology symultaneously.
	
	\item The edges of the diagram $D$ are associated with morphisms of schemes (or of diagrams of schemes) or possibly with correspondences which one knows to induce homomorphisms between the cohomology or homology spaces associated with the objects of $D$ by the representation $T$. The edges of $D$ may contain `boundary edges' as in the case of Definition \ref{defNori}. Notice that when the edges of $D$ are associated with correspondences, one does not have to care about their composites since $D$ is just a diagram.    
\end{itemize}

The representation $$T:D\to {\mathbb Q}\textrm{-vect} \quad \quad (\textrm{or }  T:D\to {\mathbb Z}\textrm{-mod})$$  can be
\begin{itemize}
	\item Betti homology or cohomology if $K$ is a subfield of $\mathbb C$,
	\item De Rham cohomology if $\textrm{char}(K)=0$,
	\item $\ell$-adic cohomology if $\ell\neq \textrm{char}(K)$,
	\item $p$-adic (crystalline or rigid) cohomology if $p=\textrm{char}(K)$,
	\item the motivic cohomology defined by Bloch's higher Chow groups. 
\end{itemize}

Applying Theorem \ref{thm:main}, we obtain:

\begin{definition}
	Let $K$ be an arbitrary base field, $R$ equal to $\mathbb Q$ or possibly to $\mathbb Z$ and $T$ a representation as above. Then the construction of Theorem \ref{thm:main} associates to these data a universal abelian $\mathbb Q$-linear (or $\mathbb Z$-linear) category ${\cal C}_{{\mathbb T}_{T}}$ endowed with a representation $$\tilde{T}:D\to {\cal C}_{{\mathbb T}_{T}}$$ and a faithful exact functor $$F_{T}:{\cal C}_{{\mathbb T}_{T}} \to {\mathbb Q}\textrm{-vect}\quad \quad (\textrm{or }  F_{T}:{\cal C}_{{\mathbb T}_{T}} \to {\mathbb Z}\textrm{-mod}).$$ 
	This category is a candidate for the category of motives over $K$.
\end{definition}

\begin{remarks}
	\begin{enumerate}[(a)]
		\item As we have already observed, the fact that the cohomological representations take their values in infinite-dimensional vector spaces over $\mathbb Q$ does not constitute a problem.
		
		\item If $T$ is a representation $D\to {\mathbb Z}\textrm{-mod}$ such that the multiplication by a prime number $p$ is invertible in all the modules which are images under $T$ of objects of $D$ then the multiplication by $p$ is automatically invertible in the category ${\cal C}_{{\mathbb T}_{T}}$. For example, if $T$ is given by the $\ell$-adic cohomology with coefficients in ${\mathbb Z}_{l}$, the multiplication by any prime number $p\neq \ell$ is automatically invertible in ${\cal C}_{{\mathbb T}_{T}}$.  
		
	\end{enumerate}
\end{remarks}

Now, each choice of $D$ and $T$ defines a candidate ${\cal C}_{{\mathbb T}_{T}}$ for the category of motives over $K$. This is clearly too much! Nonetheless, Theorem \ref{thm:independence}(ii) tells us under which conditions these candidates are equivalent to each other.

\begin{corollary}
	Let $K$ be an arbitrary base field and $R={\mathbb Q}$ or $R={\mathbb Z}$ as above. Let us fix a diagram $D$ built as above starting from the category of schemes over $K$ and consider a family $\{T\}$ of representations $$T:D\to {\mathbb Q}\textrm{-vect} \quad \quad (\textrm{or }  T:D\to {\mathbb Z}\textrm{-mod})$$ defined by ``good'' cohomological functors such as the above-mentioned ones. Then the following conditions are equivalent:
	\begin{enumerate}[(i)]
		\item The cohomological functors $T\in \{T\}$ factor through a category of motives in the usual sense, that is, there exists a $\mathbb Q$-linear (or $\mathbb Z$-linear) abelian category $\cal M$, endowed with a representation $D\to {\cal M}$, such that each $T\in \{T\}$ factors as the composite of $D\to {\cal M}$ with an exact and faithful functor 
		$${\cal M}\to {\mathbb Q}\textrm{-vect} \quad \quad (\textrm{or }  {\cal M}\to {\mathbb Z}\textrm{-mod}).$$
		
		\item The categories ${\cal C}_{{\mathbb T}_{T}}$ associated with the different $T\in \{T\}$, endowed with the representation $\tilde{T}:D\to {\cal C}_{{\mathbb T}_{T}}$, are equivalent.
		
		\item The regular theories ${\mathbb T}_{T}$ of the representations $T\in \{T\}$ are identical.   
	\end{enumerate}
	
	\begin{proof}
		The equivalence of (ii) and (iii) follows from Theorem \ref{thm:independence}(ii). The fact that (ii) implies (i) is clear. To show that (i) implies (iii) it suffices to observe that for every exact and faithful functor $$F:{\cal M}\to {\mathbb Q}\textrm{-vect} \quad \quad (\textrm{or } F: {\cal M}\to {\mathbb Z}\textrm{-mod}),$$ two subobjects $V$ and $V'$ of a product of objects $V_{1}\times \cdots \times V_{n}$ in $\cal M$ verify the relation $V\subseteq V'$ if and only if $F(V)\subseteq F(V')$.
	\end{proof}
\end{corollary} 

We can describe the above-mentioned corollary by saying that if (mixed) motives actually exist then they have a logical nature; indeed, what the different cohomological functors must have in common with each other in order for motives to exist is their associated regular theories. If this is the case, the category of motives can be built from any of these cohomological functors (independently from all the others).

We note that the identity of the regular theories associated with two representations $T$ and $T'$ of $D$ implies
\begin{itemize}
	\item the identity of the vanishing conditions of the spaces associated by $T$ and $T'$ to a given object of $D$;  
	
	\item more generally, the identity of the vanishing conditions of the subspaces defined as the kernels of the homomorphisms associated by $T$ and $T'$ to a linear combination of composites of edges in $D$, for instance the subspaces $\textrm{Ker}(T(P(u)))$ and $\textrm{Ker}(T'(P(u)))$ associated with an endomorphism $u$ and a polynomial $P$ with coefficients in $\mathbb Q$. 
\end{itemize} 

On the other hand, the regular theories of the representations $T$ do not say anything on the dimensions of the spaces and subspaces. To take care of dimensions, one needs a richer syntax, for instance that of geometric or finitary first-order logic. 

\vspace{0.5cm}

\textbf{Acknowledgements:} The second named author and the third named author are grateful to Panagis Karazeris for pointing out that the proof that the category ${\cal C}_{{\mathbb T}_{T}}$ is abelian could be shortened by invoking a classical theorem of Tierney. 

The second named author has been supported by a Marie Curie INdAM-COFUND fellowship whilst preparing this work.

\vspace{1cm}

\textsc{Luca Barbieri-Viale}

{\small \textsc{Dipartimento di Matematica ``Federigo Enriques'', Universit\`a degli Studi di Milano, via Cesare Saldini 50, 20133 Milano, Italy.}\\
	\emph{E-mail address:} \texttt{luca.barbieri-viale@unimi.it}}

\vspace{0.5cm}

\textsc{Olivia Caramello} 

{\small \textsc{UFR de Math\'ematiques, Universit\'e de Paris VII, B\^atiment Sophie Germain, 5 rue Thomas Mann, 75205 Paris, France}\\
	\emph{E-mail address:} \texttt{olivia@oliviacaramello.com}}

\vspace{0.5cm}

\textsc{Laurent Lafforgue}

{\small \textsc{Institut des Hautes \'Etudes Scientifiques, 35 route de Chartres, 91440, Bures-sur-Yvette, France}\\
	\emph{E-mail address:} \texttt{laurent@ihes.fr}}

\end{document}